\documentclass[11pt]{article}
 \usepackage{url}
 \usepackage{color}
 \usepackage{parskip}
\usepackage{amsmath,amssymb}
 \usepackage{amsthm}
\usepackage[colorlinks=true,linkcolor=black,urlcolor=black]{hyperref}
 \usepackage{algorithm,algpseudocode}
\usepackage[labelfont=bf, justification=justified]{caption}
\usepackage{todonotes}
\usepackage{amssymb}
\usepackage{csvsimple}
 \usepackage[top=3cm,bottom=3cm,left=3cm,right=3cm]{geometry}
\usepackage{pgfplots}
\pgfplotsset{compat=newest}       
\usepgfplotslibrary{groupplots}
\usepackage{pgfplotstable}
\usepackage{bbm}
\newcounter{tikzsubfigcounter}[figure]
\renewcommand{\thetikzsubfigcounter}{\the\numexpr\value{figure}+1\relax\alph{tikzsubfigcounter}}

\newcommand{\tikztitle}[1]{ %
	\refstepcounter{tikzsubfigcounter}
	\textbf{(\alph{tikzsubfigcounter})}\space\space #1 
}

\newcounter{tikzsubfigcounterinvisible}[figure]
\renewcommand{\thetikzsubfigcounterinvisible}{\the\numexpr\value{figure}+1\relax\alph{tikzsubfigcounterinvisible}}

\newcommand{\settikzlabel}[1]{ %
	\refstepcounter{tikzsubfigcounterinvisible} \label{#1} 
}

\newcommand{\nsamples}{K}
\newcommand{\convexsum}{\sum \limits_{\idx \in \idxset} \weight{\sampleidx}}
\newcommand{\statsol}[1]{\mu_{#1}}

\newcommand{\approxstatsol}[1]{\mu_{#1}^{\nsamples,h}}
\newcommand{\convexstatsol}[1]{\mu_{\sampleidx,#1}^\ast}
\newcommand{\smoothapprox}[1]{\hat{\mu}_{#1}^\nsamples}
\newcommand{\smoothapproxinitial}{\hat{\bar{\mu}}^\nsamples}
\newcommand{\sampleidx}{\idx}
\newcommand{\dirac}[1]{\delta_{#1}}
\newcommand{\wasserstein}[1]{\ensuremath{W}_{#1}}
\newcommand{\initial}{\bar{\mu}}
\newcommand{\starinitial}{\overline{\mu}^{\ast}}
\newcommand{\cF}{\mathcal{F}}
\newcommand{\dmu}[1]{\mathrm{d}\mu_{\sampleidx,#1}^\ast(\sol)}
\newcommand{\dmubar}{\mathrm{d}\bar{\mu}_{\sampleidx}^\ast(\bar{\sol})}
\newcommand{\residual}{\mathcal{R}^{st}}
\newcommand{\residualsample}[1]{\mathcal{R}^{st}_{#1}}
\newcommand{\reconstruction}{\solution^{st}}
\newcommand{\reconstsample}[1]{\solution^{st}_{#1}}
\newcommand{\initialreconstsample}[1]{\bar{\solution}^{st}_{#1}}
\newcommand{\weight}[1]{w_{#1}}

\newcommand{\flux}{f}
\newcommand*\samethanks[1][\value{footnote}]{\footnotemark[#1]}

\newcommand{\sstint}[1]{\int \limits_0^\finaltime \int \limits_{\cF} \int \limits_{\physicalspace}
#1 ~ \dx \dmu{\timevar} \dt}
\newcommand{\sstinttime}[2]{\int \limits_0^{#2} \int \limits_{\cF} \int \limits_{\physicalspace}
#1 ~ \dx \dmu{\timevar} \dt}
\newcommand{\sstinttimemidle}[3]{\int \limits_0^{#2} #3 
\int \limits_{\cF} \int \limits_{\physicalspace} #1 ~ \dx \dmu{\timevar} \dt}
\newcommand{\stinttime}[2]{\int \limits_0^{#2} \int \limits_{\physicalspace} #1 ~ \dx \dt}

\newcommand{\sstintmujan}[2]{\int \limits_0^\finaltime \int \limits_{#2} \int \limits_{\physicalspace}
#1 ~ \dx \mathrm{d}\mu_{\timevar}(\sol) \dt}

\newcommand{\ssint}[1]{\int \limits_{\cF} \int \limits_{\physicalspace}
#1 ~ \dx \dmubar}
\newcommand{\ssinttime}[2]{\int \limits_{\cF} \int \limits_{\physicalspace}
#1 ~ \dx \dmu{#2}}

\newcommand{\ssintmujan}[2]{\int \limits_{#2} \int \limits_{\physicalspace}
#1 ~ \dx \mathrm{d}\bar{\mu}(\bar{\sol})}

\newtheorem{theorem}{Theorem}[section]
\newtheorem{definition}[theorem]{Definition}
\newtheorem{lemma}[theorem]{Lemma}

\newtheorem{remark}[theorem]{Remark}

\newcommand{\solution}{u}
\newcommand{\sol}{u}
\newcommand{\test}{v}
\newcommand{\numsol}{u_h}
\newcommand{\numsample}[1]{\sol_{h,#1}}
\newcommand{\initialsol}{\bar{\sol}}
\DeclareMathOperator{\D}{D}
\newcommand{\statespace}{\mathcal{U}}
\newcommand{\compactset}{\mathcal{C}}
\newcommand{\ncls}{\ensuremath{m}}

\newcommand{\dgpolydeg}{p}
\newcommand{\idx}{\mathrm{k}}
\newcommand{\idxset}{\mathcal{K}}
\newcommand{\idxtwo}{\mathrm{m}}
\newcommand{\idxsettwo}{\mathcal{M}}

\newcommand{\physicalspace}{\ensuremath{D}}
\newcommand{\finaltime}{\ensuremath{T}}
\newcommand{\entropy}{\eta}
\newcommand{\entropyflux}{\ensuremath{q}}
\newcommand{\hessian}[2]{\ensuremath{H}_{#1}#2(#1)}
\newcommand{\hessianjan}[3]{\ensuremath{H}_{#1}#2(#3)}

\newcommand{\lowereta}{C_{\underline{\eta},\mathcal{C}}}
\newcommand{\uppereta}{C_{\overline{\eta},\mathcal{C}}}
\newcommand{\fluxconstant}{C_{\overline{\flux},\mathcal{C}}}

\newcommand{\deriv}[1]{\partial_{#1}}
\newcommand{\R}{\mathbb{R}}
\newcommand{\N}{\mathbb{N}}
\renewcommand{\P}{\mathbb{P}}
\newcommand{\cB}{\mathcal{B}}
\newcommand{\cM}{\mathcal{M}}

\newcommand{\cP}{\mathcal{P}}
\newcommand{\domain}{\ensuremath{D}}
\DeclareMathOperator{\dbracei}{[\![}
\DeclareMathOperator{\dbraceo}{]\!]}
\newcommand{\1}[1]{\mathbbm{1}_{#1}}
\newcommand{\Lp}[2]{L^{#1}(#2)}
\newcommand{\Lptwo}[3]{L^{#1}(#2;#3)}

\newcommand{\Ck}[3]{C^{#1}(#2;#3)}
\newcommand{\lpidx}{\ensuremath{r}}

\newcommand{\timevar}{\ensuremath{t}}
\newcommand{\x}{\ensuremath{x}}
\newcommand{\dx}{\mathrm{d}\x}
\newcommand{\dt}{\mathrm{d}\timevar}

\newcommand{\T}{\mathcal{T}}
\newcommand{\meshidx}{l}
\newcommand{\tidx}{n}
\newcommand{\ncells}{N_s}
\newcommand{\ntcells}{N_t}
\newcommand{\cell}{I}
\newcommand{\DG}[1]{V_h^{#1}}
\renewcommand{\t}[1]{\ensuremath{t}_{#1}}
\newcommand{\deltat}[1]{\Delta \t{#1}}
\newcommand{\numflux}{F}
\newcommand{\rkorder}{R}
\newcommand{\stages}{S}
\newcommand{\limiter}{\Pi_h}
\newcommand{\stageidx}{s}

\newcommand{\detresidual}{\mathcal{E}^{\text{det}}}
\newcommand{\detinitial}{\mathcal{E}_0^{\text{det}}}
\newcommand{\stochinitial}{\mathcal{E}_0^{\text{stoch}}}

\newcommand{\secref}[1]{Section~\ref{#1}}

\newcommand{\thmref}[1]{Theorem~\ref{#1}}

\newcommand{\defref}[1]{Definition~\ref{#1}}

\newcommand{\figref}[1]{Figure~(\ref{#1})}
\newcommand{\tabref}[1]{Table~\ref{#1}}
\newcommand{\algoref}[1]{Algorithm~\ref{#1}}

\newlength{\figureheight}
\newlength{\figurewidth}
\begin{document}
\title{Error control for statistical solutions}
\author{
Jan Giesselmann\thanks{Department of Mathematics, 
TU Darmstadt, Dolivostra\ss e 15,  64293 Darmstadt, Germany.} 
\and Fabian Meyer\thanks{Institute of Applied Analysis and Numerical Simulation, 
University of Stuttgart, 
Pfaffenwaldring 57, 70569 Stuttgart, Germany.
\newline F.M., C.R. thank the Baden-W{\"u}rttemberg Stiftung for support via the project ''SEAL``. J.G. thanks 
the German Research Foundation (DFG) for support of the project via DFG grant GI1131/1-1. $^\#$fabian.meyer@mathematik.uni-stuttgart.de}~$^{,\#}$
\and Christian Rohde\samethanks }
\date{\today}
\providecommand{\keywords}[1]{{\textit{Key words:}} #1\\ \\}
\providecommand{\class}[1]{{\textit{AMS subject classifications:}} #1}
\maketitle
\raggedbottom
\begin{abstract} \noindent
Statistical solutions have recently been introduced as a an alternative solution framework for 
hyperbolic systems of conservation laws.
In this work we derive a novel a posteriori error estimate in the Wasserstein distance
between dissipative statistical solutions 
and numerical approximations, which rely on so-called regularized empirical measures.
The error estimator can be split into deterministic parts which correspond to spatio-temporal
approximation errors and a stochastic part which reflects the stochastic
error.
We provide numerical experiments which examine the scaling 
properties of the residuals and verify their splitting.
\end{abstract}
\class{Primary: 35L65, 65M15; Secondary: 65M60, 65M700}

\section{Introduction}
Analysis and numerics of hyperbolic conservation laws have seen a significant shift of paradigms in the last decade. 
The investigation and approximation of entropy weak solutions was state of the art for a long time.
This has changed due to two reasons. Firstly, analytical insights \cite{Chiodaroli2014,DeLellisSzekelyhidi2010} revealed that weak entropic solutions to the Euler equations in several space dimensions are not unique. 
Secondly, numerical experiments \cite{FjordholmKaeppeliMishraTadmor2017,Glimm2015} have shown that in simulations of e.g.~the Kelvin-Helmholtz  instability, numerical solutions do not converge under mesh refinement. In contrast, when families of simulations with slightly varying initial data are considered, averaged quantities like 
mean, variance and also higher moments are observed to converge under mesh refinement.
This has led to several weaker (more statistics inspired) notions of solutions being proposed.
We would like to mention dissipative measure valued solutions \cite{FeireislLukacova2019} and statistical solutions \cite{FjordholmLyeMishraWeber2019}.
Considering measure valued solutions has a long history in hyperbolic conservation laws and can be traced back to the works of DiPerna \cite{DiPerna1985} considering Young measures.
Statistical solutions
are time-parametrized probability measures on spaces of integrable functions and have
been introduced recently  for scalar problems in \cite{FjordholmLanthalerMishra2017} and for systems in \cite{FjordholmLyeMishraWeber2019}.
We would like to mention that in the context of the incompressible Navier-Stokes equations and turbulence, statistical solutions have a long history going back to the seminal work of 
Foias et al., see \cite{Foias2001} and references therein.

The precise definition of statistical solutions is based on an equivalence theorem \cite[Theorem 2.8]{FjordholmLyeMishraWeber2019} that
relates probability measures on spaces of integrable functions and {\it correlation measures} on the state space.
Correlation measures are measures that determine joint probability distributions of some unknown quantity at any finite collection of spatial points.
In this sense, statistical solutions contain more information than e.g. dissipative measure valued solutions and, indeed, for scalar problems uniqueness of entropy dissipative statistical solutions can be proven.
This proof is similar to the classical proof of uniqueness of  weak solutions for scalar problems.
In contrast, for systems in multiple space dimensions the non-uniqueness of entropy weak solutions immediately implies non-uniqueness of dissipative measure valued solutions and statistical solutions.
Still, all these concepts satisfy weak-strong uniqueness principles, i.e., as long as a Lipschitz continuous solution exists in any of these classes it is the unique solution in any of these classes.
The (technical) basis for obtaining weak-strong uniqueness results is based on the 
relative entropy framework of Dafermos and DiPerna \cite{Dafermos2016}, which can be extended to dissipative statistical solutions as in \cite{FjordholmLyeMishraWeber2019}.

Convergence of numerical schemes for nonlinear systems of hyperbolic conservation laws is widely open (except when the solution is smooth). 
Exceptions are the one dimensional situation, where convergence of the Glimm scheme is well known \cite{Glimm1965} and is, indeed, used for constructing the standard Riemann semigroup \cite{Bressan1995}.
For multi-dimensional problems, there is recent progress showing convergence of numerical schemes towards dissipative measure valued solutions \cite{FeireislLukacova2019} with more information on the convergence in case the limit is an entropy weak solution.
It seems impossible to say anything about convergence rates in this setting due to the multiplicity of entropy weak solutions.

The work at hand tries to complement the a priori analysis from \cite{FjordholmLyeMishraWeber2019} with a reliable a posteriori error estimator,
i.e., we propose a computable upper bound for the numerical approximation error of statistical solutions. 
This extends results for entropy weak solutions of deterministic and random systems of hyperbolic conservation laws 
\cite{GiesselmannDedner16,GiesselmannMeyerRohdeSG19} towards statistical solutions.
One appealing feature of our error estimator is that it can be decomposed into different parts corresponding to space-time and stochastic errors.
Our analysis relies on the relative entropy framework of Dafermos and DiPerna \cite{Dafermos2016}
extended to statistical solutions.
We would like to mention that there are also other frameworks for providing a posteriori error analysis for hyperbolic systems, namely \cite{Adjerid2002}, \cite{HartmannHouston2002} and
\cite{Laforest2004} for one-dimensional systems.

The structure of this work is as follows: \secref{sec:prelim} reviews the
notion of (dissipative) statistical solutions of hyperbolic conservation laws from \cite{FjordholmLyeMishraWeber2019}. \secref{sec:scheme} is concerned with the 
numerical approximation of dissipative statistical solutions using empirical measures.
Moreover, we recall a reconstruction procedure which allows us to define the so-called regularized
empirical measure.
In \secref{sec:apost}, we present our main a posteriori error estimate between 
a dissipative statistical solution and its numerical approximation using 
the regularized empirical measure.
Section \ref{jansfavouritediscussion} discusses why defining statistical solutions to general systems like the Euler equations is not straightforward.
The technical difficulties in defining statistical solutions for general systems vanish when attention is restricted to solutions that take values in some compact subset of the state space.
We explain in Section \ref{jansfavouritediscussion} that our a posterirori error analysis directly extends to such solutions.
Finally, \secref{sec:numerics} provides some numerical experiments examining and verifying the convergence and splitting of the error estimators.

\section{Preliminaries and notations} \label{sec:prelim}

We consider the following one-dimensional system of $\ncls\in \N$ nonlinear conservation
laws:
\begin{align} \label{eq:conslaw} 
\begin{cases}
\partial_t \sol(t,x) + \partial_x  \flux(\sol(t,x))= 0,
 &( t,x)  \in   (0,T) \times \domain,
\\ \sol(0,x)= \bar{\sol}(x),  &x \in 
 \domain.
\end{cases}
\end{align} 
Here, $u(t,x)\in \statespace \subset \R^\ncls$ is the vector of conserved quantities,
$\mathcal{U}$ is an open and convex set that is called {\it state space}, 
$\flux: \statespace \to \R^m$ is the flux function, $\domain \subset \R$ is the 
spatial domain and $T \in \R_+$. We restrict ourselves to the case where $D=(0,1)$ with periodic boundary conditions.
The system \eqref{eq:conslaw} is called hyperbolic
if for any $u \in \statespace$ the flux Jacobian $\D F(u)$ has $\ncls$ real eigenvalues and admits a basis of eigenvectors.
We assume that \eqref{eq:conslaw} is equipped with an entropy/entropy flux pair
$(\entropy,\entropyflux)$, where the strictly convex function $\entropy \in \Ck{2}{\statespace}{\R}$ 
and $\entropyflux \in \Ck{2}{\statespace}{\R}$ satisfy $\D \entropyflux = \D \entropy \D\flux.$

Most literature on numerical schemes for hyperbolic conservation laws focuses on computing numerical approximations of  entropy weak solutions of \eqref{eq:conslaw}. In contrast, we are 
interested in computing statistical solutions. 
We recall the definition of statistical solutions in subsection \ref{ssec:statsol}. 
It is worthwhile to note that statistical solutions were only defined for systems for which $\statespace =\R^m$ in \cite{FjordholmLyeMishraWeber2019}.
We restrict our analysis to this setting in Sections \ref{ssec:statsol} and \ref{sec:apost}. 
In Section \ref{jansfavouritediscussion}, we discuss why we believe defining statistical solutions to general systems is not straightforward.

\subsection{Statistical solutions}\label{ssec:statsol}
In this section, we provide a brief overview of the notions needed to define statistical solutions for \eqref{eq:conslaw},
following the exposition in \cite{FjordholmLyeMishraWeber2019}, referring to 
\cite[Section 2]{FjordholmLyeMishraWeber2019} for more background, details and proofs.

Let us introduce some notation: For any topological space $X$ let $\cB(X)$ denote the Borel $\sigma$-algebra on $X$ and $\cM(X)$
denotes the set of signed Radon measures on $(X,\cB(X))$. In addition, $\cP(X)$ denotes the set of probability
measures on $(X,\cB(X))$, i.e., all non-negative $\mu \in \cM(X)$ satisfying $\mu(X)=1.$
We consider $\statespace = \R^m$ and choose $p\in [1,\infty)$ minimal, such that
$$|\flux(\sol)|,|\entropy(\sol)|,|\entropyflux(\sol)|\leq C(1+ |\sol|^p), \quad \forall  \sol \in \statespace$$
holds for some constant $C>0$. 
The following classical theorem states the duality between  $\Lp{1}{D^k; C_0(\statespace^k)}$
and $ L^\infty (D^k ; \cM(\statespace^k))$.
\begin{theorem}\label{thm:dualspace}[\cite{Ball1989}, p. 211]
For any $k \in \N$ the dual space of $\mathcal{H}_0^k(D,\statespace):= \Lp{1}{D^k; C_0(\statespace^k)}$ is 
$\mathcal{H}_0^{k*}(D,\statespace) := L^\infty (D^k ; \cM(\statespace^k))$, i.e.,  the space
of bounded, weak*-measurable maps from $D^k$ to $\cM(\statespace^k)$
under the duality pairing
\[ \langle \nu^k , g \rangle _{\mathcal{H}^k} := \int_{D^k} \langle \nu^k_x , g(x) \rangle \, \operatorname{d}\!x := \int_{D^k} \int_{\statespace^k} g(x)(\xi)\, d\nu^k_x(\xi) \, \operatorname{d}\!x .\]
\end{theorem}

\begin{definition}\label{def:corrmeasure}[Correlation measures \cite[Def 2.5]{FjordholmLanthalerMishra2017}]
 A $p$-summable correlation measure is a family $\boldsymbol{\nu}=(\nu^1, \nu^2,\dots )$ with $\nu^k \in \mathcal{H}_0^{k*}(D;\statespace) $ satisfying for all $k \in \N$
 the following properties:
 \begin{enumerate}
  \item[(a)] $\nu^k$ is a Young measure from $D^k$ to $\statespace^k$.
   \item[(b)] \textit{Symmetry}: If $\sigma$ is a permutation of $\{1,\dots,k\} $, i.e., for $x=(x_1,\ldots,x_k)$, we have
   $\sigma(x):=(x_{\sigma(1)}, \ldots, x_{\sigma(k)})$ and if 
   $f\in C_0(\statespace^k)$  then $\langle \nu_{\sigma(x)}^k,f(\sigma(\xi)) \rangle =\langle \nu_x^k,f(\xi) \rangle$ for a.e. $x \in D^k$.
   \item[(c)] \textit{Consistency}: If $f \in C_b(\statespace^k)$ is of the form $f(\xi_1,\dots,\xi_k)= g(\xi_1,\dots,\xi_{k-1})$ for some $g \in C_0(\statespace^{k-1})$ then $\langle \nu^k_{x_1,\dots,x_k}, f\rangle = \langle \nu^{k-1}_{x_1,\dots,x_{k-1}}, g\rangle$
   for a.e. $x=(x_1,\dots,x_k)\in D^k$.
   \item[(d)] $L^p$-\textit{integrability}: $ \int_{D^k} \langle \nu_x^1,|\xi|^p \rangle \, \operatorname{d}\!x <\infty$.
   \item[(e)] \textit{Diagonal continuity}:   
  $ \lim_{r \searrow 0} \int_D \frac{1}{|B_r(x_1)|} \int_{B_r(x_1)} \langle \nu^2_{x_1,x_2}, |\xi_1 - \xi_2|^p\rangle \, \operatorname{d}\!x_2 \operatorname{d}\!x_1=0.$
 \end{enumerate}
Let $\mathcal{L}^p(D;\statespace)$ denote the set of all $p$-summable correlation measures.
\end{definition}

\begin{theorem}\label{thm:duality}[Main theorem on correlation measures \cite[Thm. 2.7]{FjordholmLanthalerMishra2017}]
For every correlation measure $\boldsymbol{\nu} \in \mathcal{L}^p(D;\statespace) $ there exists a unique probability measure $\mu \in \cP(\Lp{p}{D;\statespace})$ whose $p$-th moment is finite, i.e., $\int_{L^p} \| u \|_{\Lp{p}{D;\statespace}} d \mu(u) < \infty$
 and such that $\mu $ is dual to $\boldsymbol{\nu}$, i.e.,
 \[ \int_{D^k} \langle \nu^k_x , g(x) \rangle \, \operatorname{d}\!x  =\int_{L^p} \int_{D^k} g(x,u(x)) \operatorname{d}\!x d\mu(u) \quad \forall g \in \mathcal{H}_0^k(D,\statespace), \ \forall k \in \N.\]
 Conversely, for every $\mu \in \cP(\Lp{p}{D;\statespace})$ with finite $p$-th moment there is a $\boldsymbol{\nu} \in \mathcal{L}^p(D;\statespace) $  that is dual to $\mu$.
\end{theorem}

To take into account  the time-dependence in \eqref{eq:conslaw} the authors of \cite{FjordholmLyeMishraWeber2019}
suggest to consider time parametrized probability measures.
For $T \in (0,\infty]$ consider time parametrized measures $\mu : [0,T) \rightarrow \cP(\Lp{p}{D;\statespace})$. Note that such a $\mu$ does not contain any information about correlation between function values at different times. We denote $\mu$ evaluated at time $t$ by $\mu_t$.
Let us define $\mathcal{H}_0^{k*}([0,T),D;\statespace) := L^\infty([0,T)\times D^k; \cM(\statespace^k))$ and notice that it was shown in \cite{FjordholmLyeMishraWeber2019} that it is meaningful to evaluate an element $\nu^k \in \mathcal{H}_0^{k*}([0,T),D;\statespace)$
at almost every $t \in [0,T).$

\begin{definition}
 A time-dependent correlation measure $\boldsymbol{\nu}$ is a collection $\boldsymbol{\nu}=(\nu^1, \nu^2,\dots )$ of maps $\nu^k \in \mathcal{H}_0^{k*}([0,T),D;\statespace)$ such that
 \begin{enumerate}
  \item[(a)] $(\nu_t^1, \nu_t^2 , \dots ) \in \mathcal{L}^p(D;\statespace)$ for a.e. $ t \in [0,T)$.
  \item[(b)] $L^p$ integrability:  
  \[ \operatorname{ess ~sup}\limits_{t \in [0,T)} \int_D \langle \nu_{t,x}^1, |\xi|^p \rangle \, \operatorname{d}\!x < \infty\].
  \item[(c)] \textit{Diagonal continuity}
  \[ \lim_{r \searrow 0}  \int_0^{T'} \int_D \frac{1}{|B_r(x_1)|} \int_{B_r(x_1)} \langle \nu^2_{t,x_1,x_2}, |\xi_1 - \xi_2|^p\rangle \, \operatorname{d}\!x_2 \operatorname{d}\!x_1\, dt =0 \quad \forall T' \in (0,T).\]
 \end{enumerate}
We denote the space of all time-dependent $p$-summable correlation measures
by \\ $\mathcal{L}^p([0,T),D;\statespace)$.
\end{definition}

A time-dependent analogue of Theorem \ref{thm:duality} holds true:
\begin{theorem}\label{thm:dualtime}
For every time-dependent correlation measure $\boldsymbol{\nu} \in \mathcal{L}^p([0,T),D;\statespace)$ there is a unique  (up to subsets of $[0,T)$ of Lebesgue measure zero) map $\mu :[0,T) \rightarrow \cP(\Lp{p}{D;\statespace})$ such that
 \begin{enumerate}
  \item[(a)] the mapping\[ t \mapsto \int_{L^p} \int_{D^k} g(x,u(x)) \, \operatorname{d}\!x d\mu_t(u)\]
  is measurable for all $g \in \mathcal{H}_0^k(D;\statespace)$.
  \item[(b)] $\mu $ is $L^p$-bounded, i.e.,
  \[ \operatorname{ess~sup}_{t \in [0,T)} \int_{L^p} \| u \|_{\Lp{p}{D;\statespace}} d \mu_t(u) <\infty.\]
  \item[(c)] \textit{Duality}, i.e. $\mu$ is dual to $\boldsymbol{\nu}$,
  \begin{align*}
   \int_{D^k} \langle \nu^k_{t,x} , g(x) \rangle \, \operatorname{d}\!x  =\int_{L^p} \int_{D^k} g(x,u(x)) \operatorname{d}\!x d\mu_t(u) \hfill \text{ for a.e. } t \in [0,T), \\  \hfill \forall g \in \mathcal{H}_0^k(D,\statespace), \ \forall k \in \N.   
\end{align*}
 \end{enumerate}
Conversely, for every $\mu :[0,T) \rightarrow \cP(\Lp{p}{D;\statespace})$ satisfying (a) and (b) there is a unique correlation measure $\boldsymbol{\nu} \in \mathcal{L}^p([0,T),D;\statespace)$ such that (c) holds.
\end{theorem}

\begin{definition}[Bounded support]
We say some $\bar \mu \in \cP(\Lp{p}{D;\statespace})$ has bounded support, provided there exists $C>0$ so that
\[  \| u \|_{\Lp{p}{D;\statespace}} \leq C \qquad \text{ for } \bar \mu -\text{a.e. }u \in \Lp{p}{D;\statespace}.\]
\end{definition}

\begin{definition}[Statistical solution]
Let $\bar \mu \in \cP(\Lp{p}{D;\statespace})$ have bounded support.
A statistical solution of \eqref{eq:conslaw} with initial data $\bar \mu$ is a time-dependent map $\mu :[0,T) \rightarrow \cP(\Lp{p}{D;\statespace})$ 
such that each $\mu_t$ has bounded support  and such that the corresponding correlation measures $\nu^k_t$ satisfy
\begin{equation}\label{eq:marginalevol} \partial_t \langle \nu^k_{t,x}, \xi_1\otimes \dots \otimes \xi_k\rangle + \sum_{i=1}^k  \partial_{x_i} \langle \nu^k_{t,x_1,\dots,x_k}, \xi_1 \otimes \dots \otimes \flux(\xi_i) \otimes \dots \otimes \xi_k \rangle =0\end{equation}
in the sense of distributions, for every $k \in\N$.
\end{definition}

\begin{lemma}
Let $\bar \mu \in \cP(\Lp{p}{D;\statespace})$ have bounded support.
 Then, every statistical solution $\mu: [0,T) \rightarrow  \cP(\Lp{p}{D;\statespace}) $ to  \eqref{eq:conslaw} with initial data $\bar \mu$ satisfies
 \begin{align}
  \int_0^T \int_{\Lp{p}{D;\statespace}} \int_D u(x) \partial_t \phi(t,x) &+ \flux(u(x)) \partial_x \phi(t,x) \operatorname{d}\!x d \mu_t(u) dt  \notag 
  \\ 
  &+ \int_{\Lp{p}{D;\statespace}} \int_D \bar u(x) \phi(0,x) \operatorname{d}\!x d\bar \mu(\bar u)=0,
 \end{align}
for any $\phi \in \Ck{\infty}{[0,T) \times D}{\R^m}$, where $\mu_t$ denotes $\mu$ at time $t$. 
\end{lemma}

\begin{proof}
 This is a special case of \cite[Equation 3.7]{FjordholmLyeMishraWeber2019} for $M=1$.
\end{proof}

In order to show uniqueness (for scalar problems) and weak-strong uniqueness the authors of
\cite{FjordholmLanthalerMishra2017} and \cite{FjordholmLyeMishraWeber2019} require a comparison principle that compares
statistical solutions  to convex combinations of Dirac measures. 
Therefore, the {\it entropy condition} for statistical solutions has two parts.
The first imposes stability under convex decompositions and the second is reminiscent of the standard entropy condition for deterministic problems.
To state the first condition, we need the following notation: 
For $\mu \in \cP(\Lp{p}{D;\statespace})$, $\nsamples \in \N$, $\alpha \in \R^\nsamples $ 
with $\alpha_i \geq 0$ and $\sum_{i=1}^\nsamples \alpha_i =1$ we set
\begin{align}\label{def:Lambda}
 \Lambda(\alpha,\mu) := \left\{ (\mu_1,\dots,\mu_\nsamples) \in \left( \cP(\Lp{p}{D;\statespace}) \right)^\nsamples \, : \, \sum_{i=1}^\nsamples \alpha_i \mu_i =\mu\right\}.
\end{align}
The elements of $\Lambda(\alpha,\mu)$ are strongly connected to {\it transport plans} that play a major role in defining the Wasserstein distance, see Remark \ref{rem:Lambda} for details.
Now, we are in position to state the selection criterion for statistical solutions.

\begin{definition}[Dissipative statistical solution] \label{def:dissstatsol}
A statistical solution of \eqref{eq:conslaw} is called a dissipative 
statistical solution if
\begin{enumerate}
  \item for every choice of coefficients $\alpha_i \geq 0$, satisfying 
  $\sum_{i=1}^\nsamples \alpha_i =1$ and for every 
  $(\initial_1,\ldots,\initial_\nsamples) \in \Lambda(\alpha,\initial)$,
  there exists a function 
  $t \mapsto (\mu_{1,t},\ldots,\mu_{\nsamples,t}) \in \Lambda(\alpha,\mu_t)$, such that
  each $\mu_i: [0,T) \to \cP(\Lptwo{p}{\domain}{\statespace})$ is a statistical solution
  of \eqref{eq:conslaw} with initial measure $\initial_i$.
  \item it satisfies
  \begin{multline*}
    \sstintmujan{ \Big(\entropy(\sol(x))\deriv{t} \phi(t,x) 
    + \entropyflux(\sol(x)) \deriv{x} \phi(t,x)  \Big)}{\Lp{p}{D;\statespace}}\\
    + \ssintmujan{\entropy(\bar{\sol}(x)) \phi(0,x)}{\Lp{p}{D;\statespace}} \geq 0,
  \end{multline*}
  for any $\phi \in C_c^\infty([0,\finaltime)\times \physicalspace;\R_+)$.
\end{enumerate}
\end{definition}

This selection criterion implies weak-(dissipative)-strong uniqueness, i.e., as long as some initial value problem admits a statistical solution supported on finitely many classical solutions,
then this is the only dissipative statistical solution (on that time interval), \cite[Lemma 3.3]{FjordholmLyeMishraWeber2019}. That result is a major ingredient in the proof of our main result Theorem \ref{thm:apoststatsol} and it is
indeed the special case of Theorem \ref{thm:apoststatsol} for $R^{st}_k \equiv 0.$ Both results are extensions of the classical relative entropy stability framework going back to seminal works of Dafermos and DiPerna.
The attentive reader will note that on the technical level there are some differences between \cite[Lemma 3.3]{FjordholmLyeMishraWeber2019} and Theorem \ref{thm:apoststatsol}. This is due to the following consideration:
If $L^2$ stability results are to be inferred from the  relative entropy framework, upper and lower bounds on the Hessian of the entropy and an upper bound on the Hessian of the flux $F$ are needed.
To this end, Fjordholm et.al. restrict their attention to systems for which such bounds exist globally, while we impose conditions \eqref{eq:key}, \eqref{eq:key2} and 
discuss situations in which they are satisfied (including the setting of \cite[Lemma 3.3]{FjordholmLyeMishraWeber2019}), see Remark \ref{rem:eqkey}.

\section{Numerical approximation of statistical solutions}\label{sec:scheme}
This section is concerned with the description of the 
numerical approximation of statistical solutions. Following \cite{FjordholmLyeMishra2018,FjordholmLyeMishraWeber2019}
the stochastic discretization relies on a Monte-Carlo, resp. collocation approach. 
Once samples are picked the problem at each sample is deterministic and we approximate it using the Runge--Kutta Discontinuous Galerkin method which we outline briefly. 
Moreover, we introduce a Lipschitz continuous reconstruction of the numerical solutions
which is needed for our a posteriori error estimate in \thmref{thm:apoststatsol}. Let us start with the deterministic 
space and time discretization.
\subsection{Space and time discretization: Runge--Kutta Discontinuous Galer-kin method} \label{sec:rkdg}
We briefly describe the space and time discretization of \eqref{eq:conslaw}, using the Runge--Kutta Discontinuous Galerkin (RKDG) method from
\cite{Cockburn1998}.
Let $\T:= \{ \cell_\meshidx\}_{\meshidx=0}^{\ncells-1}$, $\cell_\meshidx:=(x_\meshidx,x_{\meshidx+1})$ be a quasi-uniform triangulation of 
$\domain$. We set $h_\meshidx:= (x_{\meshidx+1}-x_\meshidx)$, $h_{\max}:= \max \limits_\meshidx h_\meshidx$, $h_{\min}:= \min \limits_\meshidx h_\meshidx$ 
for the spatial mesh and the (spatial) piecewise polynomial DG spaces for $\dgpolydeg \in \N_0$ are defined as
\begin{align*}
\DG{\dgpolydeg}:= \{v: \domain \to \R^m~|~ v\mid_{\cell} \ \in \P_\dgpolydeg(\cell,\R^m),~ \text{ for all } \cell\in \T \}.
\end{align*} 
Here $\P_\dgpolydeg$ denotes the space of polynomials of degree at most $\dgpolydeg $
and  $\mathcal{L}_{\DG{\dgpolydeg}}$ denotes the $L^2$-projection into the DG space $\DG{\dgpolydeg}$.
After spatial discretization of \eqref{eq:conslaw} we obtain the following 
semi-discrete scheme for the discrete solution $\numsol \in \Ck{1}{[0,T)}{\DG{\dgpolydeg}}$:
\begin{align}\tag{DG} 
\begin{cases} \label{eq:semidiscrete}
\sum\limits_{\meshidx=0}^{\ncells-1} \int \limits_{x_\meshidx}^{x_{\meshidx+1}}  \partial_t  \numsol \cdot \psi_h~\mathrm{d}x  
&= \sum\limits_{\meshidx=0}^{\ncells-1} \int \limits_{x_\meshidx}^{x_{\meshidx+1}} L_h(\numsol) \cdot   \psi_h~\mathrm{d}x, 
\quad \forall \psi_h \in \DG{\dgpolydeg}, \\
\numsol(t=0)&=\mathcal{L}_{\DG{\dgpolydeg}} \bar{\sol}, 
\end{cases}
\end{align}
where
$ L_h :\DG{\dgpolydeg} \to \DG{\dgpolydeg}$
is defined by
\begin{align*}
 \sum_{\meshidx=0}^{\ncells-1} \int \limits_{x_\meshidx}^{x_{\meshidx+1}}  L_h(v_h) \cdot \psi_h~\mathrm{d}x
 = &
 \sum_{\meshidx=0}^{\ncells-1} \int \limits_{x_\meshidx}^{x_{\meshidx+1}} \flux(v_h) \cdot \partial_x  \psi_h~\mathrm{d}x     
\\ &-\sum_{\meshidx=0}^{\ncells-1} \numflux(v_h(x_\meshidx^-),v_h(x_\meshidx^+) )\cdot  \dbracei \psi_h \dbraceo_\meshidx,
\qquad \forall v_h, \psi_h \in \DG{\dgpolydeg}.
\end{align*}
Here, $\numflux : \R^m \times \R^m \to \R^m$ denotes a consistent, conservative and locally Lipschitz-continuous numerical  flux,
$\psi_h(x^{\pm}):= \lim \limits_{s \searrow 0} \psi_h(x \pm s)$ are spatial traces and 
$\dbracei \psi_h \dbraceo_\meshidx:= (\psi_h(x_\meshidx^-)-\psi_h(x_\meshidx^+))$ are jumps.

The initial-value problem \eqref{eq:semidiscrete} is advanced in time by a $\rkorder$-th order
strong stability preserving Runge--Kutta (SSP-RK) method \cite{Ketcheson2008,ShuOsher1988}.
To this end, we let $0=\t{0}<\t{1}<\ldots< \t{\ntcells}=T$ be a (non-equidistant)  temporal decomposition of $[0,T]$.
We define 
$\deltat{\tidx} :=( \t{\tidx+1}-\t{\tidx})$, $\Delta t :=\max \limits_\tidx ~\Delta \t{\tidx}$.
To ensure stability, the explicit time-stepping scheme has to obey the CFL-type condition
\begin{align*}
 \Delta t \leq C \frac{ h_{\min}}{\lambda_{\text{max}}(2\dgpolydeg +1)},
\end{align*}
where $\lambda_{\text{max}}$ is an upper bound for absolute values of eigenvalues of the flux Jacobian $\D \flux$ and 
$C\in (0,1]$. 
Furthermore, we let $\limiter : \R^m \to \R^m$ be the 
TVBM minmod slope limiter from \cite{CockburnShu2001}.
The complete $\stages$-stage time-marching algorithm for given $\tidx$-th time-iterate 
$\numsol^\tidx:=\numsol(\t{\tidx},\cdot)\in \DG{\dgpolydeg}$ can then be written as follows.
\begin{algorithm}[H] \label{algo:tvbmrk}
\caption{TVBM Runge--Kutta time-step}
\begin{algorithmic}[1]
\State Set $\numsol^{(0)}$ = $\numsol^\tidx (\t{\tidx})$.
\For{$\stageidx=1,\ldots, \stages$}
\State Compute: $ \numsol^{(\stageidx)} =  \limiter\Big( \sum_{l=0}^{\stageidx-1} \delta_{\stageidx l} w_h^{\stageidx l}\Big),
\quad w_h^{ \stageidx l}=\numsol^{(l)} + \frac{\beta_{\stageidx l}}{\delta_{\stageidx l}} \Delta t_n L_h(\numsol^{(l)}).$
\EndFor
\State Set $\numsol^{\tidx+1}(\t{\tidx+1}) = \numsol^{(\stages)}.$
\end{algorithmic}
\label{algo:TVBMRungeKutta}
\end{algorithm}
Note that the initial condition $\numsol(t=0)$ is also limited by $\limiter$.
The parameters $\delta_{\stageidx l}$ satisfy the conditions $\delta_{\stageidx l}\geq 0$, $\sum_{l=0}^{\stageidx-1} \delta_{\stageidx l}=1$ ,
and if $\beta_{\stageidx l} \neq 0$, then $\delta_{\stageidx l} \neq 0$ for all $\stageidx=1,\ldots, S$, $l=0,\ldots,\stageidx$.
\subsection{Reconstruction of numerical solution}\label{ssec:reconstruction}
Our main a posteriori error estimate \thmref{thm:apoststatsol} is based on the relative entropy
method of Dafermos and DiPerna \cite{Dafermos2016} 
and its extension to statistical solution as described 
in the recent work \cite{FjordholmLyeMishraWeber2019}.
The a posteriori error estimate requires that the approximating 
solutions are at least Lipschitz-continuous in space and time.
To ensure this property, we reconstruct the numerical solution
$ \{ \numsol^{\tidx} \}_{\tidx=0}^{\ntcells }$
to a Lipschitz-continuous function in space and time. 
We do not elaborate upon this reconstruction procedure, to keep notation short and simple,
but refer to \cite{GiesselmannDedner16,GiesselmannMeyerRohdeSC2019,GiesselmannMeyerRohdeSG19}.
The reconstruction process provides a computable space-time reconstruction 
$\reconstruction \in W_\infty^1((0,T);\DG{\dgpolydeg+1}\cap C^0(\domain)) $,
which allows us to define the following residual.

\begin{definition}[Space-time residual]
We call the function 
$\residual  \in \Lptwo{2}{(0,T)\times \domain}{\R^m}$, defined by
\begin{align}\label{eq:STResidual}
\residual:= \partial_t \reconstruction + \partial_x \flux(\reconstruction),
\end{align}
the space-time residual for $\reconstruction$.
\end{definition} 
We would like to stress that the mentioned reconstruction procedure does not only render  $\reconstruction$ 
Lipschitz-continuous, but it is also specifically designed to ensure that the residual, defined in \eqref{eq:STResidual}, 
has the same decay properties as the error of the RKDG numerical scheme as $h$ tends to zero,
cf. \cite{GiesselmannDedner16}.


\subsection{Computing the empirical measure}\label{ssec:empirical}
Following \cite{FjordholmLyeMishra2018,FjordholmLyeMishraWeber2019}, 
we approximate the statistical solution
of \eqref{eq:conslaw} using empirical measures. 
In this work, we allow for arbitrary sample points and weights.
The sample points can either be obtained by randomly sampling the initial measure (Monte-Carlo),
or by using abscissae of corresponding orthogonal polynomials. 
In this article we focus on the Monte-Carlo sampling.
Let us assume that the sampling points are indexed by the  set $\idxset=\{1,...,\nsamples\}$, 
$\nsamples \in \N$ and let
us denote the corresponding weights by $\{\weight{\sampleidx}\}_{\sampleidx \in \idxset}$,
i.e., for  Monte-Carlo sampling we have $\weight{\sampleidx}= \frac{1}{\nsamples}$, for all
$\idx \in \idxset$.
We use the following Monte-Carlo type algorithm from \cite{FjordholmLyeMishraWeber2019}
 to compute the empirical measure.
 
\begin{algorithm}[H]
\linespread{1.}\selectfont{}
\caption{Algorithm to compute the empirical measure}
\begin{algorithmic}[1]
\State Given: initial measure $\initial \in \cP(\Lp{p}{\domain})$
\State For some probability space $(\Omega,\sigma,\P)$ let $\bar{\sol}:\Omega \to \Lp{p}{\domain;\statespace}$ be a random field such that the law of $\bar{u}$ with respect to $\P$ 
coincides with $\initial$
\State Draw $K$ samples from $\bar{\sol}$ and denote the samples by 
$\{\bar{\sol}_\idx \}_{\idx \in \idxset}$ 
\State{For every $\idx \in \idxset$ compute a numerical approximation 
$\{(\numsample{\idx})^n\}_{n=0}^{\ntcells}$
of \eqref{eq:conslaw} with initial data $\bar{\sol}_\idx$ using 
\algoref{algo:TVBMRungeKutta} }
\State Set 
$\approxstatsol{\finaltime}:= \sum_{\idx \in \idxset} 
\weight{\idx} \dirac{\numsample{\idx}(\finaltime)}$
\end{algorithmic}
 \label{algo:empiricalmeasure}
\end{algorithm}

\begin{definition}[Empirical measure]
Let $\{\numsample{\idx}(t)\}_{\idx \in \idxset}$ be a sequence of approximate
solutions of \eqref{eq:conslaw} with initial data $\bar{\sol}_\idx$, at time $t \in (0,T)$.
We define the empirical measure at time  $t \in (0,T)$ via
\begin{align}
  \approxstatsol{t}:= \sum_{\idx \in \idxset} \weight{\idx} \dirac{\numsample{\idx}(t)}.
\end{align}
\end{definition}
For the a posteriori error estimate in \thmref{thm:apoststatsol},
we need to consider a regularized measure, which we define 
using the reconstructed numerical solution  $\reconstruction$ obtained from the 
reconstruction procedure described in \secref{ssec:reconstruction}.
\begin{definition} [Regularized empirical measure] \label{def:regem}
Let $\{\reconstsample{\idx}(t)\}_{\idx \in \idxset}$ be a sequence of reconstructed
numerical approximations at time $t \in (0,T)$. 
We define the regularized empirical measure as follows
\begin{align}
\smoothapprox{t}=\sum_{\idx \in \idxset} 
  \weight{\idx}\dirac{\reconstsample{\idx}(t)}.
\end{align}
\end{definition}
Common metrics for computing  distances between probability measures on Banach spaces 
spaces are the Wasserstein-distances. In \thmref{thm:apoststatsol}, we bound the error between 
the dissipative statistical solution of \eqref{eq:conslaw} and the discrete regularized measure
in the 2-Wasserstein distance. 
\begin{definition}[$\lpidx$-Wasserstein distance]
Let $\lpidx\in [0, \infty)$, $X$ be a separable Banach space and let $\mu, \rho \in \cP(X)$ 
with finite $\lpidx$th moment, i.e., $\int \limits_X \|u\|_X^\lpidx~ \mathrm{d}\mu(u)< \infty$,
$\int \limits_X \|u\|_X^\lpidx~ \mathrm{d}\rho(u)< \infty$.
The $\lpidx$-Wasserstein distance between $\mu$ and $\rho$ is defined as 
\begin{align*}
  \wasserstein{\lpidx}(\mu,\rho):= 
 \Bigg( \inf \limits_{\pi \in \Pi(\mu,\rho)} 
 \int \limits_{X^2} \|u-v\|_X^\lpidx ~\mathrm{d}\pi(u,v)\Bigg)^{\frac{1}{\lpidx}},
\end{align*}
where $\Pi(\mu,\rho) \subset \cP(X^2)$ is the set of all transport plans from 
$\mu$ to $\rho$, i.e., the set of measures $\pi$ on $X^2$  with marginals $\mu$ and $\rho$, i.e.,
\begin{align*}
\pi(A \times X) = \mu(A), \qquad \pi(X \times A) = \rho(A),
\end{align*}
for all measurable subsets $A$ of $X$.
 \end{definition}
 
 \begin{remark}\label{rem:Lambda}
  It is important to recall from \cite[Lemma 4.2]{FjordholmLanthalerMishra2017} that if $\rho \in \cP(X)$  is $M$-atomic, 
  i.e., $\rho= \sum_{i=1}^\nsamples \alpha_i \delta_{u_i}$ for some $u_1, \dots, u_\nsamples \in X$ and some $\alpha_1,\dots, \alpha_\nsamples \geq 0$ with $\sum_{i=1}^\nsamples \alpha_i =1,$ then
  there is a one-to-one correspondence between transport plans in $\Pi(\mu,\rho)$ and elements of  $\Lambda(\alpha,\mu)$, which was defined in \eqref{def:Lambda}.
 \end{remark}

\section{The a posteriori error estimate}\label{sec:apost}
In this section, we present the main a posteriori error estimate between dissipative 
statistical solutions and regularized numerical approximations.
A main ingredient for the proof of \thmref{thm:apoststatsol} is the notion of relative 
entropy. 
Since the relative entropy framework is essentially an $L^2$-framework and due to the quadratic 
bounds in \eqref{eq:key}, \eqref{eq:key2} it is natural to consider the case $p=2$. 
Therefore, for the remaining part of this paper we denote $\cF:= \Lp{2}{\domain;\statespace}.$
\begin{definition}[Relative entropy and relative flux] \label{def:relentropy}
Let $(\entropy,\entropyflux)$ be a strictly convex entropy/entropy flux pair of 
\eqref{eq:conslaw}. We define the relative entropy 
$\entropy(\cdot|\cdot) : \statespace \times \statespace \to \R$, 
 and relative flux
$ \flux(\cdot|\cdot) : \statespace \times \statespace \to \R^\ncls            $
via 
\begin{align*}
  & \entropy(\sol|\test):= \entropy(\sol) - \entropy(\test) - \D \entropy(\test)(\sol -\test), \\
  &  \flux(\sol|\test) :=\flux(\sol) -\flux(\test)- 
\D \flux (\test)(\sol -\test) ,
\end{align*}
for all $\sol,\test \in \statespace$.
\end{definition}
We can now state the main result of this work which is an a posteriori error estimate
between dissipative statistical solutions and their numerical approximations using
the regularized empirical measure.
\begin{theorem} \label{thm:apoststatsol}
  Let $\statsol{t}$ be a dissipative statistical solution of \eqref{eq:conslaw} as in 
  \defref{def:dissstatsol} and let 
  $\smoothapprox{t}$
  be the regularized empirical measure from \defref{def:regem}. 
  Assume that there exist constants $A,B>0$ (independent of $h$, $M$) such that for all $t \in [0,T)$
  \begin{align}
   &| u(x) - v(x)|^2 + \flux(u(x)|v(x))  \leq A \entropy(u(x)|v(x)) \leq A^2 | u(x) - v(x)|^2
   ,\label{eq:key} \\
  & |(u(x)-v(x))^\top H_v \entropy(v(x))(u(x)-v(x)) | \leq  B |u(x)-v(x)|^2  \label{eq:key2}
  \end{align}
  $\text{for a.e. } x\in \domain, \  \text{for } \statsol{t}-\text{a.e. } u\in \cF, 
  \ \text{for }  \smoothapprox{t}-\text{a.e. } v \in \cF$. Here $H_v$ denotes the Hessian matrix
  w.r.t. $v$.\\
  Then 
  the distance between $\statsol{t}$  and $\smoothapprox{t}$ satisfies
  \begin{align*}
  \wasserstein{2}(\statsol{s}, \smoothapprox{s})^2 \leq &
  A\Bigg( \convexsum  \Bigg[ 
  \stinttime{|\residualsample{\idx}|^2}{s}\Bigg]  + 
  A \wasserstein{2}(\initial,\smoothapproxinitial )^2 \Bigg) \exp\Big( s  A^3 B(L  +1) \Big),
  \end{align*} 
  for a.e. $s \in (0,\finaltime)$ and
  $L:= \max \limits_{\idx\in \idxset} \|\deriv{x}\reconstsample{\idx}
  \|_{\Lp{\infty}{(0,s)\times \physicalspace}}$. Here we denoted 
  $\smoothapproxinitial:= \sum_{\idx \in \idxset} 
  \weight{\idx}\dirac{\initialreconstsample{\idx}}$.
\end{theorem}

\begin{remark}\label{rem:eqkey}
 There are several settings that guarantee validity of \eqref{eq:key}. Two of them are
 \begin{enumerate}
  \item The global bounds of the Hessians of $\flux$ and $\entropy$ assumed in \cite[Lemma 3.3]{FjordholmLyeMishraWeber2019}.
\item  Let $\compactset \Subset \statespace$ be
a compact and convex set. Let us assume that $\mu_t$ and $\smoothapprox{t}$ are (for all $t$) supported on solutions having values in $\compactset$ only.
Due to the regularity of 
$\flux, \entropy$ and the compactness of $\compactset$, 
there exist constants $0 < \fluxconstant < \infty $
and $ 0 < \lowereta <\uppereta < \infty$, such that 
\begin{align} \label{def:boundsfluxentropy}
|\sol^\top \hessian{\test}{\flux} \sol| \leq \fluxconstant |\sol|^2, \quad 
\lowereta |\sol|^2 \leq  | \sol^\top \hessian{\test}{\entropy} \sol| \leq \uppereta |\sol|^2,
\quad \forall \sol \in \R^\ncls, \test \in \compactset.
\end{align}
In this case, $A, B$ from \eqref{eq:key}, \eqref{eq:key2}  can be chosen as
\begin{equation}\label{eq:choiceofc} A= \max \{ (1 + \fluxconstant) \lowereta^{-1}, \uppereta\}, \quad B= \uppereta.\end{equation}
 \end{enumerate}
\end{remark}

\begin{remark}
 Let us briefly outline a fundamental difference between Theorem \ref{thm:apoststatsol} and \cite[Thm 4.9]{FjordholmLyeMishraWeber2019}. 
 The latter result is an \textit{a priori} estimate that treats the empirical measure (created via Monte Carlo sampling) as a random variable and infers convergence in the mean via the law of large numbers.
 Theorem \ref{thm:apoststatsol} provides an  \textit{a posteriori} error estimate that can be evaluated once the numerical scheme has produced an approximate solution.
 This process involves sampling the initial data measure $\bar \mu$ 
 and the ``quality'' of these samples enters the error estimate via the term $ \wasserstein{2}(\initial,\smoothapproxinitial )$.
\end{remark}

\begin{proof}
We recall that the weights $\{ \weight{\idx} \}_{\idx \in \idxset}$  satisfy
$\sum_{\idx \in \idxset} \weight{\idx}=1$. 
We further denote the vector of 
weights by $\vec{w}:=(\weight{1},\ldots,\weight{\nsamples})$,
and let $\starinitial=(\starinitial_1,\ldots,\starinitial_\nsamples) \in \Lambda(\vec{w},\initial)$ correspond to an optimal transport plan between $\initial$ and 
$\smoothapproxinitial$.
Because $\statsol{\timevar}$ is a dissipative statistical solution, there
exists $(\mu_{1,t}^\ast,\ldots,\mu_{\nsamples,t}^\ast) \in \Lambda(\vec{w},\statsol{t})$,
such that
\begin{align}\label{eq:defstatsol}
 \convexsum \Bigg[\sstint{\Big( \sol \deriv{t} \phi_\idx 
 & + \flux(\sol) \deriv{x}\phi_\idx \Big)  } \notag
\\ 
 & +  \ssint{\bar{\sol} \phi_\idx(0,\cdot)}\Bigg]=0,
\end{align}
for every  $\phi_\idx \in C_c^\infty([0,\finaltime)\times \physicalspace;\R^m)$,
$\idx \in \idxset$.
Recalling that each function $\reconstsample{\idx}$ is a Lipschitz-continuous solution of the perturbed 
conservation law \eqref{eq:STResidual}, considering its weak formulation yields
\begin{align} \label{eq:weakresidual}
\int \limits_0^\finaltime \int \limits_\physicalspace \Big(\reconstsample{\idx}
  \deriv{\timevar}\phi_\idx
 + \flux(\reconstsample{\idx})\deriv{\x} \phi_\idx \Big)~\dx \dt  +
 \int \limits_\physicalspace \initialreconstsample{\idx} \phi_\idx(0,\cdot) ~ \dx
 + \int \limits_0^\finaltime \int \limits_\physicalspace \residualsample{\idx} \phi_\idx
 ~\dx \dt = 0,
\end{align}
for every  $\phi_\idx \in C_c^\infty([0,\finaltime)\times \physicalspace;\R^m)$.
As $(\convexstatsol{\timevar})_{\idx \in \idxset}$ are probability measures on $\cF$,
we obtain (after changing order of integration)
\begin{align} \label{eq:strongeq}
 0  = & \convexsum \Bigg[\sstint{\Big(\reconstsample{\idx}
  \deriv{\timevar}\phi_\idx
 + \flux(\reconstsample{\idx})\deriv{\x} \phi_\idx \Big)} \notag
 \\
 & + \ssint{ \initialreconstsample{\idx}  \phi_\idx(0,\cdot) } 
 + \sstint{ \residualsample{\idx} \phi_\idx }\Bigg].
\end{align}
Subtracting \eqref{eq:strongeq} from  \eqref{eq:defstatsol}  and using
the Lipschitz-continuous test function \\ $\phi_\idx(t,x):= \D\eta(\reconstsample{\idx}(t,x)) \phi(t)$, where 
$\phi\in C_c^\infty([0,\finaltime);\R_+)$ yields
\begin{align} \label{eq:strongeqstat}
 0  = & \convexsum \Bigg[\sstint{\Big((\sol - \reconstsample{\idx})
  \deriv{\timevar} (\D\eta(\reconstsample{\idx}) \phi) } \notag
  \\ 
 & + \sstint{ (\flux(\sol)-\flux(\reconstsample{\idx})) \deriv{\x} (\D\eta(\reconstsample{\idx}) \phi)} \notag
 \\
 & + \ssint{ (\initialsol-\initialreconstsample{\idx})\D\eta(\initialreconstsample{\idx}) \phi(0) } 
 - \sstint{ \residualsample{\idx} \D\eta(\reconstsample{\idx}) \phi}\Bigg].
\end{align}
We compute the partial derivatives of $\D \entropy(\reconstsample{\idx}(t,x))\phi(t)$ using product and chain rule
\begin{align} 
    &\deriv{t} (\D \entropy(\reconstsample{\idx})\phi ) 
    =  \deriv{t}\reconstsample{\idx} \hessianjan{}{ \entropy}{\reconstsample{\idx}} \phi  
    + \deriv{t} \phi \D \entropy(\reconstsample{\idx}), \label{eq:derivone}
 \\ &\deriv{x} (\D \entropy(\reconstsample{\idx})\phi ) 
    =  \deriv{x}\reconstsample{\idx} \hessianjan{}{\entropy}{\reconstsample{\idx}} \phi. \label{eq:derivtwo}
\end{align}
Next, we multiply \eqref{eq:STResidual} by  $\D \entropy(\reconstsample{\idx})$. 
Upon using the chain rule for Lipschitz-continuous functions and the relationship 
$\D \entropyflux( \reconstsample{\idx})=\D \entropy(\reconstsample{\idx})
 \D \flux(\reconstsample{\idx})$ 
we derive the relation
\begin{align} \label{eq:strongEntropy}
 \D \entropy (\reconstsample{\idx}) \residualsample{\idx}
 = \deriv{\timevar} \entropy(\reconstsample{\idx}) 
 + \deriv{\x} \entropyflux(\reconstsample{\idx}).
\end{align}
Let us consider the weak form of \eqref{eq:strongEntropy} and integrate 
w.r.t.~$x,t$ and $\mathrm{d}\convexstatsol{\timevar}$ for $\idx \in \idxset$.
Upon changing the order of integration we have
\begin{align} \label{eq:entropyweak}
 0 = & \convexsum \Bigg[
 \sstint{\entropy(\reconstsample{\idx}) \deriv{t} \phi}   \notag
 \\ 
 & +\sstint{  \residualsample{\idx} \D \entropy(\reconstsample{\idx}) }
  + \ssint{\entropy(\initialreconstsample{\idx}) \phi(0)}\Bigg]
\end{align}
for any $\phi\in C_c^\infty([0,\finaltime);\R_+)$.
Since $\convexstatsol{\timevar}$ is a dissipative statistical solution it satisfies
\begin{align} \label{eq:dissstatsol}
0 \leq  \convexsum \Big[ \sstint{ \entropy(\sol) \deriv{t}\phi  } 
+ \ssint{\entropy(\initialsol) \phi(0)}\Big].
\end{align}
Hence, subtracting \eqref{eq:entropyweak} from \eqref{eq:dissstatsol} and using 
the definition of the relative entropy from \defref{def:relentropy} yields
\begin{align} \label{eq:diffstatsol}
0 \leq  & \convexsum \Bigg[ \sstint{ \entropy(\sol|\reconstsample{\idx}) 
\deriv{t}\phi  } \notag
 \\
&-\sstint{ \Big((\sol -\reconstsample{\idx} )
\D \entropy(\reconstsample{\idx})  \deriv{t}\phi 
 \Big) }  \notag
  \\ 
 & - \sstint{  \residualsample{\idx} \D \entropy(\reconstsample{\idx}) }   
+ \ssint{\Big (\entropy(\initialsol) - \entropy(\initialreconstsample{\idx})\Big) \phi(0)}\Bigg].
\end{align}
After subtracting  \eqref{eq:strongeqstat} from  \eqref{eq:diffstatsol} and using
\eqref{eq:derivone}, \eqref{eq:derivtwo} we are led to
\begin{align} \label{eq:difference}
0 \leq  & \convexsum \Bigg[ \sstint{ \entropy(\sol|\reconstsample{\idx}) 
\deriv{t}\phi    } \notag
 \\
&-\sstint{(\sol -\reconstsample{\idx} )
H\entropy(\reconstsample{\idx})  \deriv{t}\reconstsample{\idx} 
\phi  } \notag
\\ 
& -\sstint{\Big( \flux(\sol) -\flux(\reconstsample{\idx}) \Big) 
 H \entropy(\reconstsample{\idx})  \deriv{x}\reconstsample{\idx}  \phi  }  \notag
 \\
 &+ \ssint{ \entropy(\initialsol|\initialreconstsample{\idx}) \phi(0)}\Bigg].
\end{align}
Rearranging \eqref{eq:STResidual} yields
\begin{align}  \label{eq:timederiv}
\deriv{t} \reconstsample{\idx}
= - \D \flux (\reconstsample{\idx})\deriv{x} \reconstsample{\idx}
+ \residualsample{\idx}.
\end{align}
Plugging \eqref{eq:timederiv} into \eqref{eq:difference} and after rearranging we have
\begin{align} \label{eq:differencetwo}
0 \leq  & \convexsum \Bigg[ \sstint{ \entropy(\sol|\reconstsample{\idx}) 
\deriv{t}\phi   } \notag
\\
& -\sstint{ \flux(\sol|\reconstsample{\idx}) H \entropy(\reconstsample{\idx}) \deriv{x}\reconstsample{\idx} 
  \phi  }  \notag
 \\
&-\sstint{(\sol -\reconstsample{\idx} )\residualsample{\idx} 
H \entropy(\reconstsample{\idx}) 
\phi  } \notag
\\ 
 &+ \ssint{\entropy(\initialsol|\initialreconstsample{\idx}) \phi(0)}\Bigg].
\end{align}
Up to now, the choice of $\phi(t) $ was arbitrary. 
We fix $s>0$ and 
$\epsilon>0$ and define $\phi$ as follows
\begin{align*}
\phi(\sigma):= \begin{cases}
1 \qquad &: \sigma<s,\\
1- \frac{\sigma-s}{\epsilon} &: s<\sigma <s+\epsilon, \\
0 &: \sigma> s+\epsilon.
\end{cases}
\end{align*}
According to \thmref{thm:dualtime} (a) we have that the mapping 
\begin{align} \label{eq:mappinglebesgue}
t \mapsto  \ssinttime{ \entropy(\sol|\reconstsample{\idx}(t,\cdot))}{t}
\end{align}
is measurable for all $\idx \in \idxset$. Moreover, due to the quadratic bound on the relative entropy, cf. \eqref{eq:key},
Lebesgue's differentiation theorem states that a.e. $t \in (0,T)$ is a Lebesgue point 
of \eqref{eq:mappinglebesgue}.
Thus, letting $\epsilon \to 0$ we obtain
\begin{align} \label{eq:firstterm} 
 \convexsum \ssinttime{ & \entropy(\sol|\reconstsample{\idx}(s,\cdot))}{s}  \\
 & \leq   \convexsum \Bigg[ -\sstinttime{ \flux(\sol|\reconstsample{\idx})
  H \entropy(\reconstsample{\idx})  \deriv{x}\reconstsample{\idx} }{s}  \notag
 \\
&-\sstinttime{(\sol -\reconstsample{\idx} )\residualsample{\idx} 
H \entropy(\reconstsample{\idx}) }{s} \notag
\\ 
 &+ \ssint{\entropy(\initialsol|\initialreconstsample{\idx})}\Bigg] \notag .
\end{align}
The left hand side of \eqref{eq:firstterm} is bounded from below using \eqref{eq:key}.
The first term on the right hand 
is estimated using the $L^\infty(\physicalspace)$-norm
of the spatial derivative. We estimate the second term on the right hand side by 
Young's inequality. Finally, we apply \eqref{eq:key} and then \eqref{eq:key2} to both terms.
The last term on the right hand side is estimated using \eqref{eq:key}.
We, thus, end up with
\begin{align*} 
 & \convexsum    \Bigg[ A^{-1}\ssinttime{ |\sol-\reconstsample{\idx}(s,\cdot)|^2}{s}  \Bigg]
   \\
 & \leq  
  \convexsum  \Bigg[ 
  \sstinttimemidle{ |\sol -\reconstsample{\idx}(t,\cdot)|^2 }{s}
  {\Big(A^2B
  \|\deriv{x}\reconstsample{\idx}(t,\cdot) \|_{L^\infty(\physicalspace)}
   + A^2B\Big) } 
 \\
 &~~+   \sstinttime{|\residualsample{\idx}|^2}{s} +  
 A\ssint{|\initialsol-\initialreconstsample{\idx}|^2}  \Bigg].
\end{align*}
Upon using Gr\"{o}nwall's inequality we obtain
\begin{align*} 
\convexsum   & \Bigg[ \ssinttime{ |\sol-\reconstsample{\idx}(s,\cdot)|^2}{s}  \Bigg]
\\
 & \leq \convexsum  \Bigg[ A \Big(
  \sstinttime{|\residualsample{\idx}|^2}{s} + 
  A \ssint{|\initialsol-\initialreconstsample{\idx}|^2} \Big) \\ 
  &~~ \times \exp\Big( \int \limits_0^s A \Big(A^2B 
  \|\deriv{x}\reconstsample{\idx}(t,\cdot) \|_{L^\infty(\physicalspace)}
   + A^2 B\Big) \, \operatorname{d}\!t  \Bigg].
\end{align*}
Using  $\max \limits_{\idx \in \idxset} \|\deriv{x}\reconstsample{\idx}\|_{
\Lp{\infty}{(0,s)\times \physicalspace}}=:L$ and recalling that 
$(\starinitial_\idx)_{\idx \in \idxset}$ corresponds to
an optimal transport plan and that $(\convexstatsol{s})_{\idx \in \idxset} $ 
 corresponds to an admissible transport plan, we finally obtain
\begin{align*}
\wasserstein{2}(\statsol{s}, \smoothapprox{s})^2 \leq &
 A \Bigg( \convexsum  \Bigg[ 
  \stinttime{|\residualsample{\idx}|^2}{s}\Bigg]  + 
  A \wasserstein{2}(\initial,\smoothapproxinitial )^2 \Bigg)
 \\
 & \times \exp\Big( \int \limits_0^s
  \Big(A^3 B L + A^3 B \Big) \dt \Big).
\end{align*}
\end{proof}
To obtain an error estimate with separated bounds, i.e., bounds that quantify spatio-temporal 
and stochastic errors, respectively, we split the 2-Wasserstein error in initial data into
a  spatial and a stochastic part. Using the triangle inequality
we obtain
\begin{align} \label{ineq:splitting}
\wasserstein{2}(\initial,\smoothapproxinitial )
\leq \wasserstein{2}\Big(\initial,\sum_{\idx \in \idxset} 
  \weight{\idx}\dirac{\bar{\sol}_\idx}\Big)+ \wasserstein{2}\Big(\sum_{\idx \in \idxset} 
  \weight{\idx}\dirac{\bar{\sol}_\idx},\smoothapproxinitial\Big).
\end{align}
The first term in \eqref{ineq:splitting} is a stochastic error, which is inferred
from approximating the initial data by an empirical measure. On the other hand, the second term is essentially a  spatial approximation error. This can be seen from the following lemma.
\begin{lemma} With the same notation as in \thmref{thm:apoststatsol}, the following inequality holds.
\begin{align} \label{ineq:wasserstein}
\wasserstein{2}\Big(\sum_{\idx \in \idxset} 
  \weight{\idx}\dirac{\bar{\sol}_\idx},\smoothapproxinitial\Big)^2  
 \leq  
 \sum_{\idx \in \idxset}  \weight{\idx} 
 \|\bar{\sol}_\idx - \initialreconstsample{\idx} \|_{\cF}^2.
\end{align}
Equality in \eqref{ineq:wasserstein} holds provided spatial discretization errors are smaller than distances between samples. A sufficient condition is:
\begin{equation}\label{sahanavavatu} \|\bar{\sol}_\idx - \initialreconstsample{\idx} \|_{\cF}^2  
\leq \frac 1 2\min_{\ell \ne \idx} \|\bar{\sol}_\idx - \bar{\sol}_{\ell} \|_{\cF}^2 \quad \forall \idx \in \idxset\end{equation}
\end{lemma}
\begin{proof} 
 Recalling the definition of
 $\smoothapproxinitial= \sum_{\idx \in \idxset}\weight{\idx} \initialreconstsample{\idx}$
 and defining the transport plan 
 $\pi^* := \sum_{\idx \in \idxset} 
\weight{\idx}(\dirac{\bar{\sol}_\idx} \otimes \dirac{\initialreconstsample{\idx}})$
 yields the assertion, because 
 \begin{align*}
 \wasserstein{2}\Big(\sum_{\idx \in \idxset} 
  \weight{\idx}\dirac{\bar{\sol}_\idx},\smoothapproxinitial\Big)^2  
   \leq \int_{\cF^2} \| x- y \|_{\Lp{2}{\physicalspace}}^2    ~\mathrm{d}\pi^*(x,y) 
    = \sum_{\idx \in \idxset}  \weight{\idx} 
 \|\bar{\sol}_\idx - \initialreconstsample{\idx} \|_{\cF}^2.
 \end{align*}
 If \eqref{sahanavavatu} holds, $\pi^*$ is an optimal transport plan.
\end{proof}
\begin{remark}
In contrast to random conservation laws as considered in \cite{GiesselmannMeyerRohdeSG19}, 
the stochastic part of the error estimator of \thmref{thm:apoststatsol} is solely given by the discretization error of the initial data, i.e., 
given by the second term in \eqref{ineq:splitting},
which may be amplified due to nonlinear effects. However, there
is no stochastic error source during the evolution of the numerical approximation.
\end{remark}

\section{Extension to general systems (with $\statespace \neq  \R^\ncls$)}\label{jansfavouritediscussion}
Up to this point, we have stuck with the setting from \cite{FjordholmLyeMishraWeber2019} to stress that a posteriori error estimates can be obtained as long as
the relative entropy framework from \cite{FjordholmLyeMishraWeber2019}  is applicable.
Having said this, it is fairly clear that this setting does not cover certain systems of practical interest, e.g. the Euler equations of fluid dynamics.

This raises the question how  statistical solutions can be defined for more general systems.
On the first glance, it might seem sufficient to require $\mu_t$ to be a measure on some function space $\mathcal{F}$ so that $u \in \cF$ implies
 that $\flux(u)$ and $\eta(u)$ are integrable.
 There is, however, a more subtle (and rather technical) issue: The integrals in \eqref{eq:marginalevol}  need to be well defined and so does $ \int_{\mathcal{F}}   \int_D \flux(u) \operatorname{d}\!x d\mu(u)$.
 The question is not only whether this expression is finite, but whether 
 \begin{equation}\label{eq:defE} E: \cF \rightarrow \R, u \mapsto  E(u):= \int_D \flux(u) \operatorname{d}\!x  \end{equation} (and the analogous expression with $f$ replaced by $\eta$) is $\mu_t$-measurable.
  If $\mu_t$ is defined on the Borel $\sigma$-algebra of $\mathcal{F}$ the integral is well-defined provided
 $E $  is continuous. Whether this is the case depends on the topology on 
  $\mathcal{F}$.
We would like to present a toy example showing that some rather naive choices of topologies on $\mathcal{F}$ fail to ensure continuity of 
$ E $ in case $\statespace$ has  a boundary  and  $\flux$ (or $\eta$) is singular when approaching the boundary of $\statespace$.
Let us consider a scalar problem with  $\statespace=(0,\infty)$ and $\flux(u)=1/u $   and $D=(0,1)$ with periodic boundary conditions. Then, there exists no $r$ so that the map 
$E$ from \eqref{eq:defE}
is defined on $L^r$. The problem of $E$ not being well-defined can be fixed by restricting $E$ to
$\cF=L_\flux^r := \{ u \in L^r : \int_D  | \flux(u(x)) | \, \operatorname{d}\!x < \infty  \}$
but this does not make $E$ continuous. Indeed, any  $L^\infty$ neighborhood 
of $u(x) = x^{1/2}$  in $ L_\flux^\infty$  contains functions having arbitrarily large values of $E$, e.g. the functions $u^\epsilon$ with $\epsilon \in (0,\tfrac14)$ given by
\[ u^\epsilon ( x) =  \left\{ \begin{array}{ccc}
                        x^{1 - \epsilon/2} & \text{ for } &  x \in (0, \epsilon)\\
                        u(x) & \text{ for } &x \in (\epsilon,1)
                       \end{array} \right..\]
Thus, $E$ is not continuous on $L_\flux^\infty$and well-posedness of
$\int_{L_\flux^r} E(u) d\mu_t(u)$
is unclear. Due to continuous embeddings $L_\flux^\infty \rightarrow L_\flux^r$ for all $r$ the same argument works for any $r \in [1, \infty).$

The a posteriori error estimate in Section \ref{sec:apost} is not applicable in all situations in which a statistical solution can be defined but it requires \eqref{eq:key}, \eqref{eq:key2} .
There is one interesting, albeit somewhat restrictive, setting in which \eqref{eq:key}, \eqref{eq:key2}  holds and in which the technical difficulties in defining statistical solutions vanish:
The case that all measures under consideration are supported on some set of functions taking values in some compact subset of the state space.
This is the setting in which we will provide a definition of statistical solutions and an a posteriori error estimate.

\begin{definition}[$\mathcal{C}$-valued statistical solution]
 Let $\mathcal{C} \Subset \statespace$ be compact and convex. Let $\bar \mu \in \cP(\Lp{p}{D;\statespace})$ be supported on $\Lp{p}{D;\mathcal{C}}$.
A $\mathcal{C}$-valued statistical solution of \eqref{eq:conslaw} with initial data $\bar \mu$ is a time-dependent map $\mu :[0,T) \rightarrow \cP(\Lp{p}{D;\statespace})$ 
such that each $\mu_t$ is supported on $\Lp{p}{D;\mathcal{C}}$ for all $t \in [0,T)$ and such that the corresponding correlation measures $\nu^k_t$ satisfy
\begin{equation} \partial_t \langle \nu^k_{t,x}, \xi_1\otimes \dots \otimes \xi_k\rangle + \sum_{i=1}^k \partial_{x_i} \langle \nu^k_{t,x_1,\dots,x_k}, \xi_1 \otimes \dots \otimes f(\xi_i) \otimes \dots \otimes \xi_k \rangle =0\end{equation}
in the sense of distributions, for every $k \in\N$
\end{definition}

\begin{remark} 
Note that duality implies that if $\mu$ is supported on $\Lp{p}{D;\mathcal{C}}$ then the corresponding $\nu^k$ is supported on $\mathcal{C}^k$ (for all $k \in \N$).
\end{remark}

\begin{definition}[Dissipative $\mathcal{C}$-valued statistical solution]
 A $\mathcal{C}$-valued statistical  solution of \eqref{eq:conslaw} is called a dissipative 
 statistical solution provided the conditions from Definition \ref{def:dissstatsol} hold with
 ``statistical solution'' being replaced by `` $\mathcal{C}$-valued statistical  solution''.
\end{definition}

We immediately have the following theorem whose proof follows \textit{mutatis mutandis} from the proof of Theorem \ref{thm:apoststatsol}
\begin{theorem} 
  Let $\statsol{t}$ be a $\mathcal{C}$-valued dissipative statistical solution of \eqref{eq:conslaw} with initial data $\bar \mu$ as in 
  \defref{def:dissstatsol} and let 
  $\smoothapprox{t}$
  be the regularized empirical measure as in \defref{def:regem} and supported on functions with values in $\mathcal{C}$.
  Then 
  the difference between $\statsol{t}$  and $\smoothapprox{t}$ satisfies
  \begin{align*}
  \wasserstein{2}(\statsol{s}, \smoothapprox{s})^2 \leq &
  A\Bigg( \convexsum  \Bigg[ 
  \stinttime{|\residualsample{\idx}|^2}{s}\Bigg]  + 
  A \wasserstein{2}(\initial,\smoothapproxinitial )^2 \Bigg) \exp\Big( s  A^3 B(L  +1) \Big),
  \end{align*} 
  for a.e. $s \in (0,\finaltime)$ and
  $L:= \max \limits_{\idx\in \idxset} \|\deriv{x}\reconstsample{\idx}
  \|_{\Lp{\infty}{(0,s)\times \physicalspace}}$.
\end{theorem}

\section{Numerical experiment} \label{sec:numerics}
In this section we illustrate how to approximate the 2-Wasserstein distances 
that occur in \thmref{thm:apoststatsol}. Moreover, we examine the scaling 
properties of the estimators in \thmref{thm:apoststatsol} 
by means of a smooth solution of the one-dimensional
compressible Euler equations.
\subsection{Numerical approximation of Wasserstein distances}
We illustrate how to approximate the 2-Wasserstein distance 
on the example of $\wasserstein{2}(\initial,\sum_{\idx \in \idxset} 
\weight{\idx}\dirac{\bar{\sol}_\idx})$
from \eqref{ineq:splitting}.
For the given initial measure $\initial \in \cP(\Lp{2}{\physicalspace})$
we choose a probability space $(\Omega,\sigma,\P)$ and a random field
$\bar{u} \in \Lptwo{2}{\Omega}{\Lp{2}{\physicalspace}}$, such that
the law of $\bar{u}$ with respect to $\P$ coincides with $\bar{\mu}$.
We approximate the initial measure $\bar{\mu}$ using some empirical measure 
$\sum_{\idxtwo \in \idxsettwo} \weight{\idxtwo} \dirac{\bar{\sol}_\idxtwo}$,
for a second sample set $\idxsettwo:=\{1, \ldots,M \} $,
where the number of samples $M \gg K$ is significantly larger than the number 
of samples of the numerical approximation.
To distinguish between the two different sample sets $\idxset$ and $\idxsettwo$, we write 
$\{\bar{u}^{\idxset}_\idx\}_{\idx \in \idxset}$, $\{\bar{u}^{\idxsettwo}_\idxtwo\}_{\idxtwo \in \idxsettwo }$
and $\{ \weight{\idx}^\idxset \}_{\idx \in \idxset}$, $\{ \weight{\idxtwo}^\idxsettwo \}_{\idxtwo \in \idxsettwo}$
respectively.
Finally, we collect the weights $\{ \weight{\idx}^\idxset \}_{\idx \in \idxset}$
in the vector $\vec{w}^\idxset$ and  $\{ \weight{\idxtwo}^\idxsettwo \}_{\idxtwo \in \idxsettwo}$
in the vector $\vec{w}^\idxsettwo$. 

Computing the optimal transport $\pi^*$ (and thus the 2-Wasserstein 
distance) between the two atomic measures can be formulated as the following linear program,
cf. \cite[(2.11)]{Peyre2019}
\begin{align} \label{def:discretetransport}
\pi^* = \operatorname{arg~min} 
\limits_{\pi \in \Pi\Big(
\sum_{\idx \in \idxset} \weight{\idx}^\idxset \dirac{\bar{\sol}^\idxset_\idx}, 
\sum_{\idxtwo \in \idxsettwo} \weight{\idxtwo}^\idxsettwo \dirac{\bar{\sol}^\idxsettwo_\idxtwo} \Big) }
\sum_{k=1}^K \sum_{m=1}^M \pi_{k,m} \mathbb{C}_{k,m},
\end{align}
where $\Pi\Big(
\sum_{\idx \in \idxset} \weight{\idx}^\idxset \dirac{\bar{\sol}^\idxset_\idx},
\sum_{\idxtwo \in \idxsettwo} \weight{\idxtwo}^\idxsettwo \dirac{\bar{\sol}^\idxsettwo_\idxtwo} \Big):=
\{ \pi \in \R^{K\times M}_+~|~ \pi \1{M} = \vec{w}^\idxsettwo   \text{ and } \pi^\top \1{K} = \vec{w}^\idxset   \}$ 
denotes the set of transport plans and $ \1{n} := (1,\ldots,1)^\top \in \R^n$.
Moreover, the entries of the cost matrix $\mathbb{C}\in \R^{K\times M}_+$ are computed as 
\begin{align*}
\mathbb{C}_{k,m} = \| \bar{u}_k^\idxset- \bar{u}_m^\idxsettwo \|_{\Lp{2}{\physicalspace}}^2, \quad \text{for all } k=1,\ldots,K,~
m=1,\ldots,M.
\end{align*}
The linear program is solved using the network simplex algorithm \cite{Bonneel2011,Peyre2019}, 
implemented in the optimal transport library \texttt{ot.emd} \cite{Flamary2017} in python.
The 2-Wasserstein distance is finally approximated by
\begin{align} \label{eq:compwasserstein}
\wasserstein{2}\Big(\sum_{\idx \in \idxset} 
  \weight{\idx}\dirac{\bar{\sol}_\idx}, \initial\Big)^2 \approx 
\wasserstein{2}\Big(
\sum_{\idx \in \idxset} \weight{\idx}^\idxset \dirac{\bar{\sol}^\idxset_\idx},
\sum_{\idxtwo \in \idxsettwo} \weight{\idxtwo}^\idxsettwo \dirac{\bar{\sol}^\idxsettwo_\idxtwo}
\Big)^2.
\end{align}
where \eqref{eq:compwasserstein}
can be computed with \texttt{ot.emd2($\vec{w}^\idxset,\vec{w}^\idxsettwo,\mathbb{C}$)}.
\subsection{A numerical experiment} \label{sec:anumexample}
This section is devoted to the numerical study of the scaling properties of 
the individual terms in the a posteriori error estimate \thmref{thm:apoststatsol}.
In the following experiment we consider as instance of \eqref{eq:conslaw}
 the one-dimensional compressible Euler equations for the flow of an ideal gas, which
are given by
\begin{equation} \label{eq:euler}
\begin{alignedat}{2}
\hspace*{2cm}
\deriv{t} \rho &+ \deriv{x} m &&= 0 ,\\
\deriv{t} m &+ \deriv{x} \left(\frac{m^2}{\rho} + p\right)  &&=0,\\
\deriv{t} E &+ \deriv{x} \left( (E + p) \,\frac{m}{\rho}\right) &&  = 0,\hspace*{2cm}
\end{alignedat}
\end{equation}
where $\rho$ describes the mass density, $m$ the momentum and $E$ the energy of the gas. 
The constitutive law for pressure $p$ reads
$$p = (\gamma-1)\left(E - \frac{1}{2}\frac{m^2}{\rho}\right),$$
with the adiabatic constant $\gamma=1.4$.
We construct a smooth exact solution of  \eqref{eq:euler} by introducing 
an additional source term. 
The exact solution reads as follows
\begin{align} \label{eq:detbenchmark}
	\begin{pmatrix}
	\rho(t,x,\omega) \\
	 m(t,x,\omega)  \\
	E(t,x,\omega)  
	\end{pmatrix}=
	\begin{pmatrix}
	2+ 0.2\cdot\xi(\omega)\cdot\cos(6\pi(x-t)) \\[0.1cm]
	\Big(2+ 0.2\cdot \xi(\omega)\cdot\cos(6\pi(x-t)) \Big)\Big(1+0.2\cdot\xi(\omega)\cdot\sin(6\pi(x-t))\Big) \\
	\Big(2+ 0.2\cdot\xi(\omega)\cdot\cos(6\pi(x-t)) \Big)^2
	\end{pmatrix},
\end{align}
where $\xi \sim \mathcal{U}(0,1)$ is a uniformly distributed random variable. Moreover,
the spatial domain is $[0,1]_{\text{per}}$ 
and we compute the solution up to $T=0.2$.
We sample the initial measure $\bar{\mu}$ and the exact measure $\mu_T$ with 10000 samples.

For the remaining part of this section we introduce the notations
\begin{align} 
  & \detresidual(s):= \convexsum  \Bigg[ 
   \stinttime{|\residualsample{\idx}|^2}{s}\Bigg], \label{def:detresidual} \\ 
  &\detinitial:=\sum_{\idx \in \idxset}  \weight{\idx} 
 \|\bar{\sol}_\idx - \initialreconstsample{\idx} \|_{\cF}^2, \\ 
 &\stochinitial:= \wasserstein{2}(\initial,\sum_{\idx \in \idxset} 
  \weight{\idx}\dirac{\bar{\sol}_\idx})^2 \label{def:stochinitial}.
\end{align}
Moreover, when we refer to error, we mean the Wasserstein distance between exact and 
numerically computed density $\rho$ at time $t=T$.

As numerical solver we use the RKDG Code Flexi \cite{Hindenlang2012}.  
The DG polynomial degrees
will always be two and for the time-stepping we use a low storage SSP RK-method of order 
three as in \cite{Ketcheson2008}. The time-reconstruction is  also of order three.
As numerical fluxes we choose the Lax-Wendroff numerical flux
\begin{equation} \label{eq:LW}
  F(u,v):= f(w(u,v)),\quad w(u,v):= \frac{1}{2}\Big((u+v)+ \frac{\Delta t}{h}(f(v)-f(u))\Big),
\end{equation}
Computing $\detresidual, \detinitial, \stochinitial$ 
requires computing integrals, we approximate them via
Gau\ss-Legendre quadrature where we use seven points in each time-cell and 
ten points in every spatial cell. 

\subsubsection{Spatial refinement}
In this section we examine the scaling properties of $\detresidual(T)$ and 
$\detinitial$ when we gradually increase the number of spatial cells. We start initially
on a coarse mesh with 16 elements, i.e., $h=1/16$
and upon each refinement 
we subdivide each spatial cell uniformly into two cells.
To examine a possible influence of the stochastic resolution we consider a coarse and 
fine sampling with 100 and 1000 samples, respectively.
\tabref{table:hrefine_coarse} and \figref{fig:spatialref_coarse} display the error estimator parts, $\detresidual$, $\stochinitial$ and 
$\detinitial$ from \eqref{def:detresidual}-\eqref{def:stochinitial} when we compute the numerical 
approximation with 100 samples and \tabref{table:hrefine_fine} and \figref{fig:spatialref_fine} display the same quantities 
for 1000 samples. 
We observe that the convergence of $\detresidual$ and $\detinitial$ does not depend on the 
stochastic resolution. Indeed, both quantities decay when we increase the number of spatial elements and they are uniform with respect to changes in the sample size.
Moreover, $\detinitial$ exhibits the expected order of convergence which is six for 
a DG polynomial degree of two (note that we compute squared quantities). 
It can also be observed that $\detresidual$ converges with order five, which is due to a
suboptimal rate of convergence on the first time-step.
This issue has also been discussed in detail in \cite[Remark 4.1]{GiesselmannMeyerRohdeSG19}.
Furthermore, we see that the numerical error remains almost constant upon mesh refinement,
since it is dominated by the stochastic resolution error, which is described by $\stochinitial$.
Since $\stochinitial$ reflects the stochastic error it also remains constant upon spatial mesh refinement.
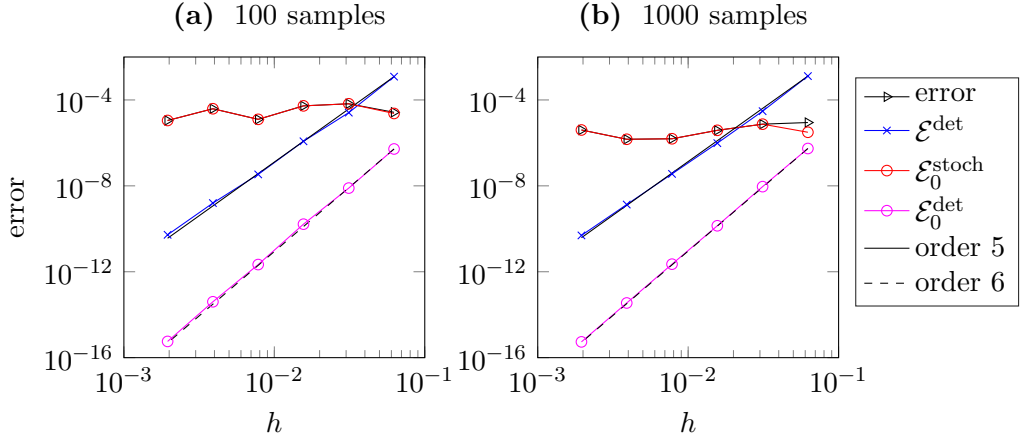
\begin{figure}
  \settikzlabel{fig:spatialref_coarse}
  \settikzlabel{fig:spatialref_fine}
%
%
\definecolor{mycolor1}{rgb}{1.00000,0.00000,1.00000}%
\begin{tikzpicture}

\begin{groupplot}[
group style={group size=2 by 1, horizontal sep = 1.5cm,  vertical sep = 2cm},
width=4cm,
height=4cm,
scale only axis,
xmode=log,
xmin=0.001,
xmax=0.1,
xminorticks=true,
ymode=log,
ymin=1e-16,
ymax=0.01,
yminorticks=true,
axis background/.style={fill=white},
]
\nextgroupplot[
title=\tikztitle{100 samples},
ylabel={error},
xlabel={$h$},
]

\addplot [color=black, mark=triangle, mark options={solid, rotate=270, black}]
  table[row sep=crcr]{%
0.0625	2.66104226043282e-05\\
0.03125	6.62264511276079e-05\\
0.015625	5.30670957019757e-05\\
0.0078125	1.25056658114709e-05\\
0.00390625	3.89273469829138e-05\\
0.001953125	1.10714429432115e-05\\
};

\addplot [color=blue, mark=x, mark options={solid, blue}]
  table[row sep=crcr]{%
0.0625	0.00121783719199383\\
0.03125	2.57104819671727e-05\\
0.015625	1.18697153288183e-06\\
0.0078125	3.41419766921585e-08\\
0.00390625	1.57546690871117e-09\\
0.001953125	5.25369778857946e-11\\
};

\addplot [color=red, mark=o, mark options={solid, red}]
  table[row sep=crcr]{%
0.0625	2.34808939127105e-05\\
0.03125	6.63617123331803e-05\\
0.015625	5.30535699874401e-05\\
0.0078125	1.25059185151764e-05\\
0.00390625	3.89272994095795e-05\\
0.001953125	1.10714419743215e-05\\
};

\addplot [color=mycolor1, mark=o, mark options={solid, mycolor1}]
  table[row sep=crcr]{%
0.0625	5.21246018083948e-07\\
0.03125	7.86538379143158e-09\\
0.015625	1.64278568026167e-10\\
0.0078125	2.16166554166052e-12\\
0.00390625	4.00047808220443e-14\\
0.001953125	5.69584025776697e-16\\
};

\addplot [color=black]
  table[row sep=crcr]{%
0.0625	0.00121783719199383\\
0.001953125	3.62943766115256e-11\\
};

\addplot [color=black, dashed]
  table[row sep=crcr]{%
0.0625	5.21246018083948e-07\\
0.001953125	4.85448183570009e-16\\
};

\nextgroupplot[
title=\tikztitle{1000 samples},
legend style={at={(1.6,0.17)}, anchor=south east,,legend cell align=left, align=left, draw=white!15!black},
xlabel={$h$},
]

\addplot [color=black, mark=triangle, mark options={solid, rotate=270, black}]
  table[row sep=crcr]{%
0.0625	8.88785265312534e-06\\
0.03125	7.45735385175831e-06\\
0.015625	3.85154048127712e-06\\
0.0078125	1.55818534167835e-06\\
0.00390625	1.47711754741208e-06\\
0.001953125	3.9934752381533e-06\\
};
\addlegendentry{error}

\addplot [color=blue, mark=x, mark options={solid, blue}]
  table[row sep=crcr]{%
0.0625	0.00129156309948217\\
0.03125	3.00152110962814e-05\\
0.015625	9.75433640685351e-07\\
0.0078125	3.60732917408606e-08\\
0.00390625	1.34780612318896e-09\\
0.001953125	4.91740720418174e-11\\
};
\addlegendentry{$\detresidual$}

\addplot [color=red, mark=o, mark options={solid, red}]
  table[row sep=crcr]{%
0.0625	3.12892891122152e-06\\
0.03125	7.44984385175831e-06\\
0.015625	3.85307586011641e-06\\
0.0078125	1.55820991846122e-06\\
0.00390625	1.47711952640374e-06\\
0.001953125	3.99347590479016e-06\\
};
\addlegendentry{$\stochinitial$}

\addplot [color=mycolor1, mark=o, mark options={solid, mycolor1}]
  table[row sep=crcr]{%
0.0625	5.48475669458395e-07\\
0.03125	9.19286613620045e-09\\
0.015625	1.38421133385901e-10\\
0.0078125	2.25406425852047e-12\\
0.00390625	3.51498474994718e-14\\
0.001953125	5.43001921615807e-16\\
};
\addlegendentry{$\detinitial$}

\addplot [color=black]
  table[row sep=crcr]{%
0.0625	0.00129156309948217\\
0.001953125	3.84915798748185e-11\\
};
\addlegendentry{order 5}

\addplot [color=black, dashed]
  table[row sep=crcr]{%
0.0625	5.48475669458395e-07\\
0.001953125	5.10807772593942e-16\\
};
\addlegendentry{order 6}

\end{groupplot}
\end{tikzpicture}%
 \caption{Spatial refinement for 100 and 1000 samples.}
 \label{fig:spatialref}
\end{figure}

\begin{table}[htb!]
\centering
\begin{filecontents*}{hrefine_coarse.csv}
h,error,residual,errorsample,errorreconst
0.0625,2.661e-05,0.0012178,2.3481e-05,5.2125e-07
0.03125,6.6226e-05,2.571e-05,6.6362e-05,7.8654e-09
0.015625,5.3067e-05,1.187e-06,5.3054e-05,1.6428e-10
0.0078125,1.2506e-05,3.4142e-08,1.2506e-05,2.1617e-12
0.0039062,3.8927e-05,1.5755e-09,3.8927e-05,4.0005e-14
0.0019531,1.1071e-05,5.2537e-11,1.1071e-05,5.6958e-16
\end{filecontents*}
\pgfplotstabletypeset[
    col sep=comma,
    string type,
    every head row/.style={%
        before row={\hline
            \multicolumn{5}{|c|}{$h$-refinement, 100 samples}  \\
            \hline
        },
        after row=\hline
    },
    every last row/.style={after row=\hline},
    columns={h,error,residual,errorsample,errorreconst},
    columns/h/.style={column name=$h$, column type={|c}},
    columns/error/.style={column name=error, column type={|c}},
    columns/residual/.style={column name=$\detresidual$, column type={|c}},
        columns/errorsample/.style={column name=$\stochinitial$, column type={|c}},
    columns/errorreconst/.style={column name=$\detinitial$, column type={|c|}},
    ]{hrefine_coarse.csv}       
    \caption{Spatial refinement for 100 samples.}
    \label{table:hrefine_coarse}
\end{table}
\begin{table}[htb!]
\centering
\begin{filecontents*}{hrefine_fine.csv}
h,error,residual,errorsample,errorreconst
0.0625,8.8879e-06,0.0012916,3.1289e-06,5.4848e-07
0.03125,7.4574e-06,3.0015e-05,7.4498e-06,9.1929e-09
0.015625,3.8515e-06,9.7543e-07,3.8531e-06,1.3842e-10
0.0078125,1.5582e-06,3.6073e-08,1.5582e-06,2.2541e-12
0.0039062,1.4771e-06,1.3478e-09,1.4771e-06,3.515e-14
0.0019531,3.9935e-06,4.9174e-11,3.9935e-06,5.43e-16
\end{filecontents*}
\pgfplotstabletypeset[
    col sep=comma,
    string type,
    every head row/.style={%
        before row={\hline
            \multicolumn{5}{|c|}{$h$-refinement, 1000 samples}  \\
            \hline
        },
        after row=\hline
    },
    every last row/.style={after row=\hline},
    columns={h,error,residual,errorsample,errorreconst},
    columns/h/.style={column name=$h$, column type={|c}},
    columns/error/.style={column name=error, column type={|c}},
    columns/residual/.style={column name=$\detresidual$, column type={|c}},
        columns/h/.style={column name=$h$, column type={|c}},
    columns/errorsample/.style={column name=$\stochinitial$, column type={|c}},
    columns/errorreconst/.style={column name=$\detinitial$, column type={|c|}},
    ]{hrefine_fine.csv}      
    \caption{Spatial refinement for 1000 samples.}
    \label{table:hrefine_fine}
\end{table}

\subsubsection{Stochastic refinement}
In this section, we consider stochastic refinement, i.e., we increase the number of
samples and keep the spatial resolution fixed. Similarly to the numerical example in 
the previous section we consider a very coarse spatial discretization with 8 elements 
(\tabref{table:stochrefine_coarse}, \figref{fig:stochref_coarse}) and a fine 
spatial discretization with 256 elements (\tabref{table:stochrefine_fine}, \figref{fig:stochref_fine}).
We observe that 
$\detresidual$ and $\detinitial$ remain constant when we increase the number of 
samples. This is the correct behavior, since  both residuals reflect 
spatio-temporal errors.
For the coarse spatial discretization we observe that the numerical error does not 
decrease when increasing the the number of samples since the spatial discretization error,
reflected by $\detresidual$, dominates 
the numerical error.
For the fine spatial discretization the numerical error is only dominated by 
the stochastic resolution error and thus the error converges with the same order.
We observe an experimental rate of convergence of one (resp.~one half after taking the square root).
Finally, we see that $\stochinitial$ is independent of the spatial discretization 
because its convergence is not altered by the spatio-temporal resolution.
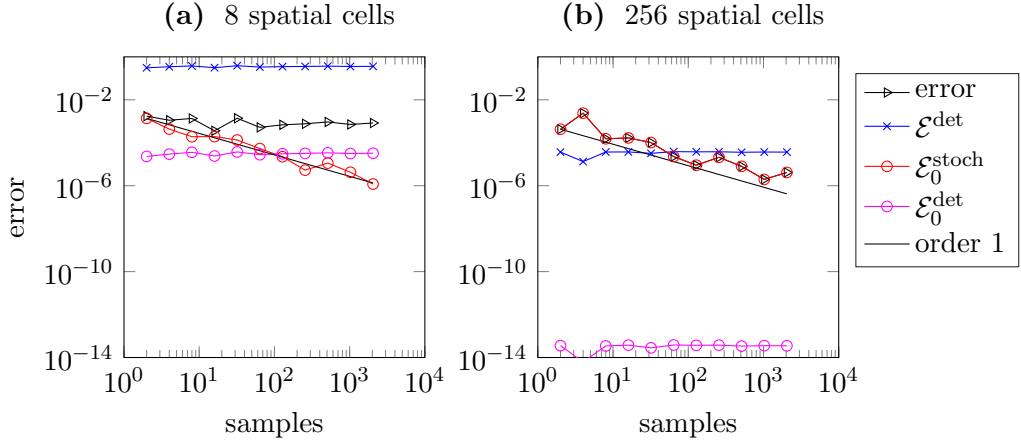
\begin{figure}
 \settikzlabel{fig:stochref_coarse}
  \settikzlabel{fig:stochref_fine}
%
%
\definecolor{mycolor1}{rgb}{1.00000,0.00000,1.00000}%
\begin{tikzpicture}
\begin{groupplot}[
group style={group size=2 by 1, horizontal sep = 1.5cm,  vertical sep = 2cm},
width=4cm,
height=4cm,
scale only axis,
xmode=log,
xmin=1,
xmax=10000,
xminorticks=true,
ymode=log,
]
\nextgroupplot[
title=\tikztitle{8 spatial cells},
ylabel={error},
xlabel={samples},
ymin=1e-14,
ymax=1,
]

\addplot [color=black, mark=triangle, mark options={solid, rotate=270, black}]
  table[row sep=crcr]{%
2	0.00172676538756674\\
4	0.00111673624521827\\
8	0.00135222015010615\\
16	0.000360846819282033\\
32	0.00138392243998916\\
64	0.000521725546981906\\
128	0.00068473207000943\\
256	0.000763415158032183\\
512	0.000930308692873636\\
1024	0.000713172554031996\\
2048	0.000830343783094905\\
};

\addplot [color=blue, mark=x, mark options={solid, blue}]
  table[row sep=crcr]{%
2	0.309554019541155\\
4	0.346375705024806\\
8	0.37836516779559\\
16	0.313110060544952\\
32	0.387577588687347\\
64	0.334538327457291\\
128	0.350877885036243\\
256	0.357701854901239\\
512	0.369205263627951\\
1024	0.356203712093963\\
2048	0.361596268607594\\
};

\addplot [color=red, mark=o, mark options={solid, red}]
  table[row sep=crcr]{%
2	0.0013644067347396\\
4	0.000432794415738105\\
8	0.000191061457990438\\
16	0.000193180211922747\\
32	0.000134341775156289\\
64	5.33675684415726e-05\\
128	2.24161170025128e-05\\
256	5.29634169508424e-06\\
512	1.12923534129487e-05\\
1024	4.22939865274445e-06\\
2048	1.19691402661534e-06\\
};

\addplot [color=mycolor1, mark=o, mark options={solid, mycolor1}]
  table[row sep=crcr]{%
2	2.30845070808604e-05\\
4	2.95223079197578e-05\\
8	3.59519634868733e-05\\
16	2.40979191894801e-05\\
32	3.74214832583436e-05\\
64	2.77498268806148e-05\\
128	3.05424125309543e-05\\
256	3.18651352626559e-05\\
512	3.39304694092756e-05\\
1024	3.1577983474331e-05\\
2048	3.25453557166899e-05\\
};

\addplot [color=black]
  table[row sep=crcr]{%
2	0.0013644067347396\\
2048	1.33242845189414e-06\\
};

\nextgroupplot[
title=\tikztitle{256 spatial cells},
xlabel={samples},
ymin=1e-14,
ymax=1,
legend style={at={(1.6,0.3)}, anchor=south east,,legend cell align=left, align=left, draw=white!15!black},
]

\addplot [color=black, mark=triangle, mark options={solid, rotate=270, black}]
  table[row sep=crcr]{%
2	0.000426580486614769\\
4	0.00238576905372699\\
8	0.000153227720119202\\
16	0.000167711656267952\\
32	0.000102239709621287\\
64	2.28806916078901e-05\\
128	9.01575657648186e-06\\
256	2.10619309691602e-05\\
512	7.96710376815913e-06\\
1024	1.96626386082958e-06\\
2048	4.13783311341525e-06\\
};
\addlegendentry{error}

\addplot [color=blue, mark=x, mark options={solid, blue}]
  table[row sep=crcr]{%
2	3.67627257231304e-05\\
4	1.33019708628757e-05\\
8	3.70376772015998e-05\\
16	3.73385148372904e-05\\
32	3.22229727256914e-05\\
64	3.86335016402972e-05\\
128	3.77541142019539e-05\\
256	3.80559854416849e-05\\
512	3.60373766511412e-05\\
1024	3.70038841870323e-05\\
2048	3.67325185822362e-05\\
};
\addlegendentry{$\detresidual$}

\addplot [color=red, mark=o, mark options={solid, red}]
  table[row sep=crcr]{%
2	0.00042658045884964\\
4	0.00238576920256541\\
8	0.000153227723791781\\
16	0.000167711624186403\\
32	0.000102239764930706\\
64	2.2880662531467e-05\\
128	9.01573884254241e-06\\
256	2.10619024088997e-05\\
512	7.9671188738358e-06\\
1024	1.96625827215431e-06\\
2048	4.13783147568438e-06\\
};
\addlegendentry{$\stochinitial$}

\addplot [color=mycolor1, mark=o, mark options={solid, mycolor1}]
  table[row sep=crcr]{%
2	3.57638728874961e-14\\
4	6.20862985886712e-15\\
8	3.45372615210264e-14\\
16	3.75820153868133e-14\\
32	2.82385556978422e-14\\
64	3.82318239142781e-14\\
128	3.69712643072684e-14\\
256	3.75665921441444e-14\\
512	3.40107788798905e-14\\
1024	3.55786942070047e-14\\
2048	3.50735144871531e-14\\
};
\addlegendentry{$\detinitial$}

\addplot [color=black]
  table[row sep=crcr]{%
2	0.00042658045884964\\
2048	4.16582479345352e-07\\
};
\addlegendentry{order 1}

\end{groupplot}

\end{tikzpicture}%
  \caption{Stochastic refinement for 8 and 256 spatial cells.}
  \label{fig:stochrefine}
\end{figure}

\begin{table}[htb!]
\centering
\begin{filecontents*}{stochrefine_coarse.csv}
h,error,residual,errorsample,errorreconst
2,1.7268e-03,3.0955e-01,1.3644e-03,2.3085e-05
4,1.1167e-03,3.4638e-01,4.3279e-04,2.9522e-05
8,1.3522e-03,3.7837e-01,1.9106e-04,3.5952e-05
16,3.6085e-04,3.1311e-01,1.9318e-04,2.4098e-05
32,1.3839e-03,3.8758e-01,1.3434e-04,3.7421e-05
64,5.2173e-04,3.3454e-01,5.3368e-05,2.775e-05
128,6.8473e-04,3.5088e-01,2.2416e-05,3.0542e-05
256,7.6342e-04,3.577e-01,5.2963e-06,3.1865e-05
512,9.3031e-04,3.6921e-01,1.1292e-05,3.393e-05
1024,7.1317e-04,3.562e-01,4.2294e-06,3.1578e-05
2048,8.3034e-04,3.616e-01,1.1969e-06,3.2545e-05
\end{filecontents*}
\pgfplotstabletypeset[
    col sep=comma,
    string type,
    every head row/.style={%
        before row={\hline
            \multicolumn{5}{|c|}{stochastic refinement, $h=1/8$}  \\
            \hline
        },
        after row=\hline
    },
    every last row/.style={after row=\hline},
    columns={h,error,residual,errorsample,errorreconst},
    columns/h/.style={column name=samples, column type={|c}},
    columns/error/.style={column name=error, column type={|c}},
    columns/residual/.style={column name=$\detresidual$, column type={|c}},
    columns/errorsample/.style={column name=$\stochinitial$, column type={|c}},
    columns/errorreconst/.style={column name=$\detinitial$, column type={|c|}},
    ]{stochrefine_coarse.csv}   
    \caption{Stochastic refinement for 8 spatial cells.}
    \label{table:stochrefine_coarse}
\end{table}

\begin{table}[htb!]
\centering
\begin{filecontents*}{stochrefine_fine.csv}
h,error,residual,errorsample,errorreconst
2,4.2658e-04,3.6763e-05,4.2658e-04,3.5764e-14
4,2.3858e-03,1.3302e-05,2.3858e-03,6.2086e-15
8,1.5323e-04,3.7038e-05,1.5323e-04,3.4537e-14
16,1.6771e-04,3.7339e-05,1.6771e-04,3.7582e-14
32,1.0224e-04,3.2223e-05,1.0224e-04,2.8239e-14
64,2.2881e-05,3.8634e-05,2.2881e-05,3.8232e-14
128,9.0158e-06,3.7754e-05,9.0157e-06,3.6971e-14
256,2.1062e-05,3.8056e-05,2.1062e-05,3.7567e-14
512,7.9671e-06,3.6037e-05,7.9671e-06,3.4011e-14
1024,1.9663e-06,3.7004e-05,1.9663e-06,3.5579e-14
2048,4.1378e-06,3.6733e-05,4.1378e-06,3.5074e-14
\end{filecontents*}
\pgfplotstabletypeset[
    col sep=comma,
    string type,
    every head row/.style={%
        before row={\hline
            \multicolumn{5}{|c|}{stochastic refinement, $h=1/256$}  \\
            \hline
        },
        after row=\hline
    },
    every last row/.style={after row=\hline},
    columns={h,error,residual,errorsample,errorreconst},
    columns/h/.style={column name=samples, column type={|c}},
    columns/error/.style={column name=error, column type={|c}},
    columns/residual/.style={column name=$\detresidual$, column type={|c}},
    columns/errorsample/.style={column name=$\stochinitial$, column type={|c}},
    columns/errorreconst/.style={column name=$\detinitial$, column type={|c|}},
    ]{stochrefine_fine.csv}       
    \caption{Stochastic refinement for 256 spatial cells.}
    \label{table:stochrefine_fine}
\end{table}

\section{Conclusions}
This work provides a first rigorous a posteriori error estimate for 
numerical approximations of dissipative statistical solutions in one spatial dimension.
Our numerical approximations rely on so-called regularized empirical measures,
which enable us to use the relative entropy
method of Dafermons and DiPerna \cite{Dafermos2016} within the framework of 
dissipative statistical solutions introduced by the authors of \cite{FjordholmLyeMishraWeber2019}.
We derived a splitting of the error estimator into a stochastic and a spatio-temporal part.
In addition, we provided a numerical example verifying this splitting. Moreover, our numerical results
confirm that the quantities that occur in our a posteriori error estimate 
decay with the expected order of convergence.

Further work should focus on a posteriori error control for multi-dimensional systems 
of hyperbolic conservation laws, especially the correct extension of the space-time 
reconstruction to two and three spatial dimensions.
Moreover, based on the a posteriori error estimate it is 
possible to design reliable space-time-stochastic numerical
schemes for the approximation of dissipative statistical solutions. 
For random conservation laws and classical $L^2(\Omega;L^2(D))$-estimates, the
stochastic part of the error estimator is given by the discretization error in the initial data and
an additional stochastic residual which occurs during the evolution, see for example \cite{GiesselmannMeyerRohdeSC2019,GiesselmannMeyerRohdeSG19}.
For dissipative statistical solutions the stochastic part of the error estimator in the 2-Wasserstein distance
is directly related to the stochastic discretization error of the initial data which may be amplified in
time due to nonlinear effects.
This result shows that stochastic adaptivity for dissipative statistical solutions becomes
significantly easier compared to random conservation laws since only
stochastic discretization errors  of the initial data (and their proliferation) need to be controlled.
The design of space-stochastic adaptive numerical schemes based on this observation will be the subject of further research. 

\bibliographystyle{siam}
\bibliography{bibliography}
\end{document}